\documentclass{article} 

\usepackage{amsfonts}
\usepackage{amssymb,amsmath,amsthm,amscd} 
\usepackage{graphicx} 
\usepackage{subfigure}
\usepackage{color}
\usepackage{longtable}

\newcounter{Mycounter}[section]
\newtheorem{Theo}[Mycounter]{Theorem} 

\newtheorem{Lem}[Mycounter]{Lemma} 
\newtheorem{Rem}[Mycounter]{Remark}

\newcommand{\setR}{\mathbb{R}} 
\newcommand{\setC}{\mathbb{C}}

\newcommand{\setN}{\mathbb{N}}

\newcommand{\Oi}{\Omega_{\rm int}} 
\newcommand{\Oe}{\Omega_{\rm ext}}
\newcommand{\smat}[1]{\left(\begin{smallmatrix} #1 \end{smallmatrix}\right)}
\DeclareMathOperator{\MT}{ {\cal M}_{\kappa_0} }
\DeclareMathOperator{\MTinv}{ {\cal M}_{\kappa_0}^{-1} }
\DeclareMathOperator{\LT}{{\cal L}}
 
\DeclareMathOperator{\OpT}{\cal T}

\newcommand{\BilF}[2]{a\!\left(#1,#2\right)}
\newcommand{\LiF}[1]{g\!\left(#1\right)}

\newcommand{\fe}{u}
\newcommand{\fz}{v}

\newcommand{\Fe}{U}
\newcommand{\Fz}{V}

\newcommand{\bfe}{\mathbf{\fe}}
\newcommand{\bfz}{\mathbf{\fz}}

\newcommand{\bFe}{\mathbf{\Fe}}

\newcommand{\El}{K}

\newcommand{\ep}{\varepsilon}
\newcommand{\BLM}{F}

\DeclareMathOperator{\Dop}{\cal D}
\DeclareMathOperator{\Iop}{\cal I}
\DeclareMathOperator{\id}{\mathrm id}

\DeclareMathOperator{\curl}{\rm \bf curl}

\DeclareMathOperator{\scurl}{\rm curl}
\DeclareMathOperator{\divo}{\rm div}
\DeclareMathOperator{\Null}{ 0}

\def\vector#1#2#3{\left(\!\!\begin{array}{c} #1 \\ #2 \\ #3 \end{array}\!\!\right)}
\def\vect#1#2{\left(\!\!\begin{array}{c} #1 \\ #2 \end{array}\!\!\right)}

\def\oversym#1#2{\mathop{#1}\limits^{#2}}
\def\RA#1{\oversym\longrightarrow{#1}}

\title{High order Curl-conforming Hardy space infinite elements for 
  exterior Maxwell problems}
\author{Lothar Nannen \and Thorsten Hohage \and Achim Sch\"adle  \and Joachim Sch\"oberl}

\date{}
        
\begin{document}

\maketitle

\begin{abstract} 
A construction of prismatic Hardy space infinite elements to
discretize wave equations on unbounded domains $\Omega$ in $H^1_{\rm loc}(\Omega)$, $H_{\rm loc}(\curl;\Omega)$ and
$H_{\rm loc}(\divo;\Omega)$ is presented.  As our motivation is to solve Maxwell's
equations we take care that these infinite elements fit into the
discrete de Rham diagram, i.e.\ they span discrete spaces, which
together with the exterior derivative form an exact sequence.
Resonance as well as scattering problems are considered in the
examples. Numerical tests indicate super-algebraic convergence
in the number of additional unknows per degree of freedom on the
coupling boundary that are required to realize the Dirichlet to Neumann map.
\end{abstract}


\section{Introduction} 

Time-harmonic Maxwell's equations are used to model resonance 
and electromagnetic scattering problems, that occur 
for example in the modelling of photolitography processes~\cite{Burger2005}, 
the simulation of meta-materials~\cite{Gansel:10}, photonic
cavities~\cite{Karletal:09} or plasmon resonances.

We consider the following second order elliptic equations on $\Omega \subset \setR^3$.
First the time-harmonic second order Maxwell system for the
electric field which in variational form in $H_{\rm loc}(\curl;\Omega)$ is given by
\begin{equation}
  \label{eq:Maxwell}
  \int_\Omega  \nabla \times \bfe \cdot \nabla \times \bfz - \ep \kappa^2\, \bfe\cdot \bfz dx 
  = \LiF{\bfz},
  \qquad \forall \bfz \in H_{\rm c}(\curl;\Omega). 
\end{equation}
$H_{\rm loc}(\curl;\Omega)$ ($H_{\rm c}(\curl;\Omega)$) is the space of vector fields $\bfz$ which are together with the curl $\nabla \times \bfz$ locally (compactly supported) in $(L^2(\Omega))^3$. And second the Helmholtz equation in variational form in $H^1_{\rm loc}(\Omega)$
\begin{eqnarray}
  \label{eq:Helmholtz}
  \int_\Omega \nabla \fe \cdot \nabla \fz - \ep  \kappa^2 \, \fe \fz dx &=& \LiF{\fz},
  \qquad \forall \fz \in H^1_{\rm c}(\Omega)
\end{eqnarray}
as well as in mixed formulation for 
$\boldsymbol{\sigma}:=\frac{1}{i \kappa} \nabla \fe \in H_{\rm loc}(\divo;\Omega)$ and $u \in L_{\rm loc}^2(\Omega)$
\begin{subequations}
  \label{eq:Helmholtzdiv}
  \begin{eqnarray}
    \int_\Omega  \fe \nabla \cdot \boldsymbol{\tau} - i \ep \kappa \, 
    \boldsymbol{\sigma} \cdot \boldsymbol{\tau} dx &=& 0, 
    \qquad \forall \boldsymbol{\tau} \in H(\divo;\Omega),
    \\
    \int_\Omega  i \kappa \fe \fz + \nabla \cdot \boldsymbol{\sigma} \fz  dx &=& \LiF{\fz}, 
    \qquad \forall \fz \in L^2(\Omega).
  \end{eqnarray}
\end{subequations}
The constant $\kappa$ with positive real part is the wavenumber, $\LiF{\fz}$
contains any source term and $\ep$ is the local permittivity or
refraction index.

We deal in this paper with two types of problems: The scattering and the resonance problem. The scattering or source problem consists of finding a solution $\bfe$, $\fe$ or $(\boldsymbol{\sigma},u)$ to \eqref{eq:Maxwell}, \eqref{eq:Helmholtz} or \eqref{eq:Helmholtzdiv} for a given wavenumber $\kappa>0$ and a given functional $g$. In the resonance problem we are looking for eigenpairs to \eqref{eq:Maxwell}, \eqref{eq:Helmholtz} or \eqref{eq:Helmholtzdiv} with $g\equiv 0$ where $\kappa$ is now a complex resonance with positive real part and $\bfe$, $\fe$ or $(\boldsymbol{\sigma},u)$ the resonance function.

It is very common that~\eqref{eq:Maxwell} and~\eqref{eq:Helmholtz},
\eqref{eq:Helmholtzdiv} are posed on an unbounded domain $\Omega$,
such that radiation conditions as the Silver-M\"uller or Sommerfeld
radiation condition are required to make the problem well-posed. For
the numerical solution of~\eqref{eq:Maxwell} or \eqref{eq:Helmholtz}
resp. \eqref{eq:Helmholtzdiv} these radiation conditions at
``infinity'' are replaced by transparent or non-reflection boundary
conditions at some artificial boundary $\Gamma$ that separates the
bounded convex computational domain $\Oi$, where one wants to
calculate a solution, from the exterior domain $\Oe$~\cite{Givoli:04,Hagstrom2003}. Usually the
exterior domain is assumed to be homogeneous without sources. For the
presented method $\ep$ is allowed to be non-constant, but there are
restrictions to $\ep$ given at the end of
Section~\ref{sec.HardyScalar}.

The ansatz to realize transparent boundary condition presented in this
article is based on the pole condition~\cite{Schmidt:02,PC1}, which
sloppy speaking says, that a solution of the Helmholtz or time-harmonic Maxwell equation in the exterior domain is radiating, if its Laplace transform
taken along a path from the boundary $\Gamma$ to infinity is
holomorphic, i.e.\ does not have any singularities in the lower
complex half plane. The pole condition is numerically realized by the
Hardy-space infinite element method by mapping the lower half plane
onto the unit disc and approximating the function there.  This concept
has already been applied successfully to the two dimensional Helmholtz
equation (\cite{HohageNannen:09,Nannen:08,NannenSchaedle:09}) and to
time-dependent one dimensional Schr\"odinger and wave equations
(\cite{RSSZ:08}). 

The Hardy space infinite element method leads to tensor products of 
special functions in generalized radial
direction and surface boundary functions. The name infinite element
method refers to the classical infinite element methods \cite{DemIhl,DemkowiczPal:98}, where the choice of radial functions is based on
the asymptotic behavior of 
Hankel functions of the first kind. There are several other approaches to realize transparent
boundary conditions based on separable coordinates and special
functions~\cite{GroteKeller:98}, perfectly matched layer
constructions~\cite{Simon:79,Berenger:94,ChewWeedon:94,ColMonk},
boundary integral approaches~\cite{HsiaoWendland:08} and local high
order approximations~\cite{Givoli:04}.

Most of these method have the drawback to depend non-linearly on the
wavenumber $\kappa$. Hence, for resonance problems they would lead to a non-linear eigenvalue
problem.  Although it is possible to solve the resulting non-linear
eigenvalue problems, see e.g.  \cite{LenoirVulliermeHazard:92}, it is
desirable to avoid them.  Therefore complex scaling methods are
currently the standard method for solving resonance problems, see
e.g.~\cite{HeinHohageKochSchoeberl:07,KimPasciak:08}. Unfortunately
these methods give rise to spurious resonance modes and several
parameters have to be optimized for each problem. The presented Hardy
space infinite element method has the advantage to depend linearly on
$\kappa^2$ and there is only one parameter and the number of degrees of
freedom in generalized radial direction to choose.

Apart from the treatment of the exterior problem, we use the high
order finite elements of Sch\"oberl and Zaglmayr
\cite{SchoeberlZaglmayr:05,Zaglmayr:06} for the bounded interior
problem. Recent overviews of $H^1(\Omega)$, $H(\curl;\Omega)$ and $H(\divo;\Omega)$
conforming method are \cite{Monk:03,Demkowiczetal:08} and in the
context of differential forms \cite{Hiptmair:02}.
 
The paper is organized as follows: Section~\ref{sec.ExtDomain} gives
the variational formulation for the scalar problem
\eqref{eq:Helmholtz} as well as for \eqref{eq:Maxwell} and
\eqref{eq:Helmholtzdiv} using Hardy space functions. In Section
\ref{sec:BasisFunc} the local basis functions for $H^1(\Omega)$, $H(\curl;\Omega)$,
$H(\divo,\Omega)$ and $L^2(\Omega)$ conforming methods are presented. In Section
\ref{sec:Sequence} the local basis functions are collected together
and in Section \ref{sec:Numerics} numerical results show the
exponential convergence of the method.

\section{The exterior problem}
\label{sec.ExtDomain}

Let $P$ be a convex polyhedron containing $\mathbb{R}\setminus \Omega$.
The unbounded domain $\Omega$ is split into a bounded interior domain
$\Oi:=\Omega \cap P$ and an unbounded exterior domain $\Oe:=\Omega \setminus P$. 
Without loss of
generality we may assume that the surface $\Gamma=\overline{\Oi} \cap
\overline{\Oe}$ is triangulated. In the following we first summarize the
the basic ideas of the Hardy space method for exterior domains $\Oe$ in one
dimension and introduce our notation. Afterwards, the extension to
three dimensions and vector-valued functions is given.

\subsection{Exterior variational formulation in one dimension}
\label{sec.Hardy1d}
The starting point of our method is the characterization of outgoing
waves by the so-called \emph{pole condition} \cite{Schmidt:02}: 
A general solution to the homogeneous one dimensional Helmholtz equation
\begin{equation}
 -\fe''(x) - \kappa^2 \fe(x) = 0, \qquad x\geq 0,
\end{equation}
is given by $u(x)=C_1 \exp( i \kappa x) + C_2 \exp(- i \kappa x)$,
whereas $\fe$ is outgoing, iff $C_2=0$. In the original form the pole
condition states that $\fe$ is outgoing, if the Laplace transformed
function $\tilde \Fe :=\LT \fe $ with $\tilde \Fe(s) =\frac{C_1}{s- i
\kappa} + \frac{C_2}{s+ i \kappa}$ has no poles with negative
imaginary part. Hence, $\tilde \Fe$ has to be analytic in the lower
half plane and belongs to the Hardy space $H^-(\setR)$.

For the definition of Hardy spaces we refer to \cite{Duren:70} or
\cite[Sec. 2.1]{HohageNannen:09}. Roughly speaking they are
$L^2$-boundary values of functions, which are holomorphic in some
region. Here, we mainly use the Hardy space $H^-(\setR)$ of the lower
complex half plane $\{s \in \setC ~|~\Im(s)<0\}$ and the Hardy space
$H^+(S^1)$ of the complex unit disk $\{ z \in \setC ~|~|z|<1\}$. A
parameter dependent M\"obius transform
\begin{equation*}
z \mapsto s(z) = i \kappa_0 \frac{z+1}{z-1}  
\end{equation*}
with $\kappa_0 \in \setC$ and $\Re(\kappa_0)>0$ is used to construct a 
mapping $\MT: H^-(\setR) \to H^+(S^1)$, which is unitary
up to a constant factor by
\begin{equation}
\MT \tilde \Fe(z) := \tilde \Fe(s(z)) \frac{1}{z-1}.
\end{equation}

The variant of the pole condition we are using in this paper is the
following: $\fe$ is outgoing, if $\Fe:=\MT \LT \fe \in H^+(S^1)$. Our
aim is to derive a transformed variational formulation of the exterior
problem, which incorporates the radiation condition by the choice of
the Hardy space $H^+(S^1)$ for the transformed solution (see eq.\
\eqref{eq:weak1dHardy} below).  First, we need the following identity,
which follows from basic complex analysis (see \cite[Lemma
A.1]{HohageNannen:09}): For suitable functions $\fe$ and $\fz$ and
their Laplace-M\"obius transforms $\Fe:=\MT \LT \fe$, $\Fz:=\MT \LT
\fz$ in $H^+(S^1)$ we have
\begin{equation}
\label{eq:BFTrans}
 \int_0^\infty \fe(\xi) \fz(\xi) ~d \xi =  
 \BilF {\Fe}{\Fz},
\end{equation}
where
\begin{equation}
\label{eq:defFB}
 \BilF{\Fe}{\Fz} := 
 \frac{-2i\kappa_0}{2\pi} \int_{S^1}  \Fe(z) \Fz(\overline{z}) |dz|.
\end{equation}
The second ingredient  required for the derivation of the transformed 
variational formulation is the transformation of the derivative operator
$\partial_{\xi}$ in propagation direction. For this end we introduce 
the operators $\OpT_\pm:\setC \times H^+(S^1) \to H^+(S^1)$ by 
\begin{equation}
\label{eq:DefOpT}
\OpT_\pm (\fe_0,\hat \Fe) (z):= \frac{1}{2} \left( \fe_0 + (z \pm 1)\hat \Fe(z) \right)\,.
\end{equation}
It is easy to check, that $\OpT_\pm$ are injective: If
$\hat \Fe(z)=\sum_{j=0}^\infty \alpha_j z^j \in H^+(S^1)$ and if the image
$\OpT_\pm (\fe_0,\hat \Fe)$ vanishes, then
$|\fe_0|=|\alpha_0|=|\alpha_1|=\dots$, and since $(\alpha_j)_j$ is 
square summable it must be identically zero as well as $\fe_0$. For solutions $\fe$ to the homogeneous Helmholtz
equation it is shown in \cite[Lemma 5.6]{Nannen:08}, that indeed $\MT
\LT \fe$ belongs to the image $\OpT_- (\setC \times H^+(S^1))$.

If $\MT \LT \fe = \frac{1}{i \kappa_0} \OpT_- (\fe_0,\hat \Fe)$, the transformation of spatial derivative $\partial_\xi$ is
given by
\begin{equation}
\label{eq:Tp}
  \MT \LT \left( \partial_\xi \fe \right) = \OpT_+ (\fe_0,\hat \Fe)
  = i\kappa_0 \OpT_+\OpT_-^{-1} \MT \LT \left( \fe \right) ,
\end{equation}
so we set
\begin{equation}
\label{eq:derxihat}
\hat \partial_\xi := i \kappa_0 \OpT_+ \OpT_-^{-1}\,.
\end{equation}

This is not the only reason for using the operator $\OpT_-$: $\OpT_-^{-1}$ decomposes the function $\MT \LT \fe$
into the boundary value $\fe_0=\fe(0)$ and a function $\hat \Fe$ with no
contribution to the boundary value of $\fe$, which is important to ensure to continuity of the finite element basis functions over the boundary $\Gamma$ (see \ Sec.\ \ref{sec:H1element}).

Using these formulas and definitions, the weak formulation of the
one-dimensional Helmholtz equation~\eqref{eq:Helmholtz} with
e.g. $\Omega=[-1,\infty)$, $\Oe=[0,\infty)$, $\Oi=[-1,0]$ 
and $\ep(x)=1$ for $x \in \Oe$ can
be transformed to
\begin{equation}
  \label{eq:weak1dHardy}
  \int_{-1}^0 \, \fe' \fz' - \ep \,\kappa^2\, \fe \fz dx + 
  \BilF{\hat \partial_\xi \Fe }{ \hat \partial_\xi \Fz} - \kappa^2 \BilF{\Fe}{\Fz}=\LiF{\fz}.
\end{equation}
For the implementation we refer to Section~\ref{sec:H1element} or for more details 
to~\cite[Sec. 2.4]{HohageNannen:09}. 

\subsection{Segmentation of the exterior domain}
\label{sec:TrafoPrism}

In three dimensions we assume that the exterior domain $\Oe$ can be discretized by
infinite non-degenerate pyramidal frustums in such a way that on each
frustum $\ep$ is constant and that there is one distance variable
$\xi$ that allows to globally parameterize $\Oe$ by blowing up
$\Gamma$.  Such a parameterization cannot be constructed locally. We
refer to~\cite{Schmidt:02,Kettner:07,Kettner:2008,Zschiedrichetal:06}
for details on constructing such a segmentation in a general two
dimensional setting.  These assumptions on $\ep$ allow to treat for
example layered media, infinite cones or waveguides.

\begin{figure}
\centering
\resizebox{0.7\textwidth}{!}{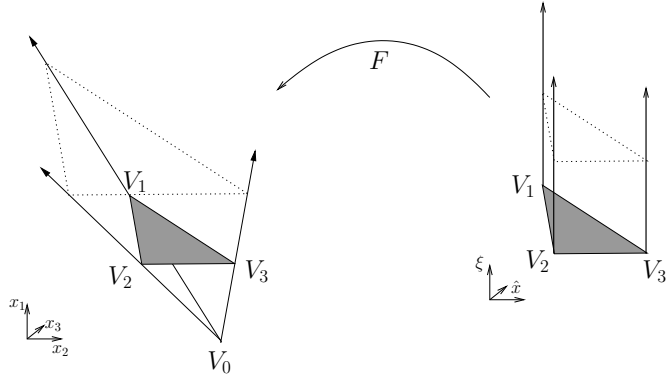} 
\caption{Transformation $\BLM$ of a right prism $\hat \El$ to a pyramidal frustum $\El$}
\label{fig:segTrafo}
\end{figure}

To keep the presentation simple we consider a special exterior
discretization: 
Suppose that we have constructed a tetrahedral mesh of
$\Oi$, which induces a triangulation $\mathcal{T}$ on $\Gamma$.
Assuming now that $\ep$ is constant in $\Oe$ we can chose an arbitrary reference
point $V_0$ in $P$ and discretize $\Oe$ by infinite pyramidal
frustums $\El$ with a triangular base $T \in \mathcal{T}$ and infinite
faces formed by rays starting from $V_0$ and proceeding through the
vertices of the surface triangulation $\mathcal{T}$:
\begin{equation}
 \label{eq:infprisma}
 \El:=\{ \hat{x}+\xi(\hat{x}-V_0) \in \setR^3 ~|~ \xi \geq 0,~\hat{x} \in T\}.
\end{equation}
$\xi$ is called a (global) generalized radial variable. 
To compute integrals over a pyramidal frustum $\El$ the parameterization 
\begin{equation}
\label{eq:BLM}
  \BLM_\El(\xi,\hat{x}):= \hat{x}+\xi(\hat{x}-V_0) 
\end{equation}
is used to transform the frustum onto a right prism
$\hat{\El}:=[0,\infty)\times T$ over the surface triangle $T$, cf. 
Fig.~\ref{fig:segTrafo}. 
Using the surface gradient
\(
  \nabla_{\hat{x}}^T 
\)
of the two dimensional surface applied to the three components of
$\hat{x} \in \setR^3$, the Jacobian of the bilinear mapping
$\BLM_\El$ is given by
\begin{equation}
  J(\xi,\hat{x})=
  \begin{pmatrix} 
    \hat{x}-V_0 ~,~ (1+\xi)~ \nabla_{\hat{x}}^T \hat{x} 
  \end{pmatrix} 
  \in \setR^{3 \times 3}.
\end{equation}
For later use it is useful to factorize 
\begin{subequations}
\label{eq:Jacobian}
  \begin{eqnarray}
    J(\xi,\hat{x})&=& \hat{J}(\hat{x}) 
    \begin{pmatrix} 1 & 0 &0 \\ 0 & 1+\xi& 0\\ 0&0&1+\xi
    \end{pmatrix}
  \end{eqnarray}
into on $\hat{x}$ dependent part     
$
\hat{J}(\hat{x}) := \begin{pmatrix} 
      \hat{x}-x_0 ~,~ \nabla_{\hat{x}}^T \hat{x} 
    \end{pmatrix}
$
and a $\xi$ dependent part and note that
\begin{eqnarray}    
    J^{-T}(\xi,\hat{x})&=& \hat{J}(\hat{x}) ^{-T} 
    \begin{pmatrix} 1 & 0 &0 \\ 0 & \frac{1}{1+\xi}& 0\\ 0&0&\frac{1}{1+\xi}
    \end{pmatrix},\\
    |J (\xi,\hat{x})| &=& (1+\xi)^2 |\hat{J}(\hat{x})|.
  \end{eqnarray}
\end{subequations}

\begin{Rem}
In~\cite{HohageNannen:09} the exterior domain is the complement of a
ball of radius $a$ around the origin. Functions $u$ in the exterior
domain are written in polar coordinates with an additional scaling
factor for symmetry reasons (cf. \cite[eq. (3.2)]{HohageNannen:09}):
\begin{equation*}
	u_{\rm ext} (r,\hat{x})= (r+1)^{(d-1)/2} u( (r+1)\hat{x}).
\end{equation*}
If the reference point $V_0$ in \eqref{eq:infprisma} is set to the
origin, if $P$ would be a ball and if we  neglect the scaling factor,
then the discretization of $\Oe$ in this paper coincides with the one 
in~\cite{HohageNannen:09}.
\end{Rem}

The canonical transformations that are used to transform basis functions $\hat \fz$ and their derivatives
from a reference element $\hat \El$ to basis functions $\fz$ defined
on a local element $\El=\BLM (\hat \El)$ are summarized in the following
Lemma:

\begin{Lem}
\label{Lem:TrafoBasis}
 Let $\El \subset \setR^3$ such that $\El=\BLM (\hat \El)$ 
with Jacobian Matrix $J = J_\BLM$ and $\nabla_{\xi,\hat{x}}:= (\partial_\xi,\nabla_{\hat{x}}^{(1)},\nabla_{\hat{x}}^{(2)})^T$.
\begin{enumerate}
 \item For $\hat \fz \in H^1(\hat \El)$ let $ \fz \circ \BLM:= \hat \fz$. Then $(\nabla \fz) \circ \BLM =  J^{-T} \nabla_{\xi,\hat{x}} \hat \fz$.
\item For $\hat \bfz \in H(\curl, \hat \El)$ let $\bfz \circ \BLM:= J^{-T} \hat \bfz$. Then $(\nabla \times \bfz) \circ \BLM =  \frac{1}{|J|} J ~ \nabla_{\xi,\hat{x}} \times  \hat \bfz$.
\item For $\hat \bfz \in H(\divo,\hat \El)$ let $\bfz \circ \BLM:=  \frac{1}{|J|} J ~ \hat \bfz$. Then $(\nabla \cdot  \bfz) \circ \BLM =  \frac{1}{|J|}  \nabla_{\xi,\hat{x}} \cdot \hat \bfz$.
\end{enumerate}
\end{Lem}

\begin{proof}
See e.g. ~\cite[Sec. 3.9]{Monk:03}.
\end{proof} 

\subsection{Exterior variational formulation in $H^1(\Oe)$}
\label{sec.HardyScalar}

To extend~\eqref{eq:weak1dHardy} to scalar functions in three dimensions, we need to transform
integrals over the pyramidal frustums $\El$ to integrals over the right prism $\hat{\El}$. 
Using the definitions and notations of the last section we have 
\begin{equation}
  \int_\El \fe  \fz ~dx =  
  \int_T \int_0^\infty (\fe \circ \BLM) ~(\fz \circ \BLM) ~|J| ~ d\xi~d\hat{x}.
\end{equation}
The Laplace and M\"obius transform is applied in $\xi$-direction to the functions $\fe$ and $\fz$, i.e.\ we consider 
\begin{subequations}\label{eq:MapScalar}
\begin{eqnarray}
  \Fe(\bullet,\hat{x})&:=&\MT \LT \left(\fe \circ \BLM(\bullet,\hat{x})\right) \in H^+(S^1),\\
  \Fz(\bullet,\hat{x})&:=&\MT \LT \left(\fz \circ \BLM(\bullet,\hat{x})\right) \in H^+(S^1)
\end{eqnarray}
\end{subequations}
for $\hat{x} \in T$. In \cite{PC1,HohageNannen:09} it was proven that this generalization of the one dimensional pole condition leads to a radiation condition equivalent to other formulations of the physically relevant radiation condition. Using~\eqref{eq:BFTrans} the infinite integral $\int_0^\infty (...) d\xi$ is 
transformed into a bilinear form in the Hardy space 
$H^+(S^1)$:
\begin{equation}
 \label{eq:trafoscalar}
 \int_\El \fe  \fz dx = 
 \int_T \BilF{ \Dop \Fe (\bullet,\hat{x})}{  \Dop \Fz (\bullet,\hat{x})}  |\hat{J}(\hat{x})| d\hat{x},
\end{equation}
where the operator $\Dop$ in \eqref{eq:trafoscalar}, which is implicitly defined by
\begin{equation}
\Dop  \left(\MT \LT \fe\right) = \MT 
\LT \left((\bullet+1)\fe \right)\, ,
\end{equation}
occurs due to the fact that 
the determinant of the Jacobi matrix includes the factor $(\xi+1)^2$.
Straightforward calculations yield
\begin{equation}
 (\Dop \Fe)(z)= \frac{(z-1)^2}{2 i \kappa_0} \left(\Fe\right)'(z) + 
 \left(1+\frac{z-1}{2 i \kappa_0} \right) \left(\Fe\right)(z)\, ,
\end{equation}
so $\Dop$ has the following matrix representation with respect to 
the monomial basis $\{z^0,z^1,z^2,\dots\}$ of $H^+(S^1)$: 
\begin{equation}
 \id + \frac{1}{2 i
    \kappa_0} {\scriptstyle 
    \left(\begin{array} {ccccc} -1 & 1 & &  \\
 1 & -3 & 2   &  \\ &2 & -5 & \ddots  \\ & & \ddots & \ddots  
      \end{array}\right)}.
\end{equation}
Due to the transformation of the gradients in Lemma \ref{Lem:TrafoBasis} we have
\begin{equation}
  \label{eq:intgradfield}
  \int_\El \nabla \fe \cdot \nabla \fz ~dx = 
  \int_{T} \int _0^\infty J^{-T} \nabla_{\xi,\hat{x}} \fe \circ \BLM \cdot 
  J^{-T} \nabla_{\xi,\hat{x}} \fz \circ \BLM~|J| d\xi d\hat{x}.
\end{equation}
%
We define the $\hat{x}$ dependent matrix $G=(g_{jk})_{j,k=1}^3:=|\hat{J}|\hat{J}^{-1}\hat{J}^{-T}$ and 
use \eqref{eq:Jacobian} to calculate
\begin{equation}
 (|J| J^{-1} J^{-T}) (\xi,\hat{x}) =\left(\!\begin{array}{ccc}
g_{11}(\hat{x})(1+\xi)^2 &g_{12}(\hat{x})(1+\xi)&g_{13}(\hat{x})(1+\xi)\\
g_{21}(\hat{x})(1+\xi) &g_{22}(\hat{x})&g_{23}(\hat{x})\\
g_{31}(\hat{x})(1+\xi) &g_{32}(\hat{x})&g_{33}(\hat{x})\\
\end{array}\!\right).
\end{equation}
Thus \eqref{eq:intgradfield} becomes
\begin{equation}
\label{eq:stiffH1}
\int_\El \nabla \fe \cdot \nabla \fz ~dx = A_{G,T} \left(\vect{\left(\Dop \hat  \partial_{\xi} \otimes \id\right) \Fe}{\left(\id \otimes \nabla_{\hat x}\right)\Fe},\vect{\left(\Dop \hat \partial_{\xi} \otimes \id\right) \Fz}{\left(\id \otimes \nabla_{\hat x}\right) \Fz} \right)
\end{equation}
with
\begin{equation}
\label{eq:BilFmdim}
 A_{G,T}  \left( (\Fe_1,...,\Fe_r)^T,(\Fz_1,...,\Fz_r)^T \right)\!:=\!\int_T\!\sum_{j,k=1}^{r} g_{jk}(\hat{x}) \BilF{ \Fe_j(\bullet,\hat{x})}{ \Fz_k (\bullet,\hat{x})}d\hat{x}.\!
\end{equation}
Since $\Fe \in H^+(S^1) \otimes L^2(\Gamma)$, we use a tensor product notation to combine the operators $\Dop$, $\hat \partial_\xi$ acting on $H^+(S^1)$ and $\nabla_{\hat x}$ acting on $L^2(\Gamma)$ to operators on $H^+(S^1) \otimes L^2(\Gamma)$ (see e.g. \cite{Murphy:90}). $\id$ denotes the identity operator. Note, that the matrix $G$ in \eqref{eq:BilFmdim} can be replaced with $|\hat{J}|$ in order to get a short notation for the right hand side of \eqref{eq:trafoscalar} as well.

Similar to \eqref{eq:weak1dHardy} a variational formulation of
\eqref{eq:Helmholtz} using the Hardy space $H^+(S^1)$ consist of the
integral over the bounded interior domain $\Oi$ and the integral over
the unbounded exterior domain $\Oe$, which is given by the sum over
all surface triangles of the integrals \eqref{eq:trafoscalar} and
\eqref{eq:stiffH1}: %
\begin{equation}
\begin{aligned}
 \LiF{\fz}& = \int_{\Oi} \nabla \fe \cdot \nabla \fz - \ep  \kappa^2 \, \fe \fz dx\\
 +& \sum_{T \in \mathcal{T}} \left( A_{G,T} \left(\vect{\left(\Dop \hat  \partial_{\xi} \otimes \id\right) \Fe}{\left(\id \otimes \nabla_{\hat x}\right)\Fe},\vect{\left(\Dop \hat \partial_{\xi} \otimes \id\right) \Fz}{\left(\id \otimes \nabla_{\hat x}\right) \Fz} \right) - A_{\ep |\hat{J}|,T} (\Fe,\Fz)\right).
\end{aligned}
\end{equation}
The coefficient function $\ep$ in can be incorporated into the exterior integrals, if it
depends for each pyramidal frustum only on the surface variable $\hat{x}$. Numerically, the integrals over the surface triangle $T$ will be approximated using a quadrature formula, while the bilinear-form $a$ can be computed analytically as presented in \cite{HohageNannen:09}.

\subsection{Exterior variational formulation in $H(\curl;\Oe)$}
\label{sec:HSMVF}
Since a solution $\bfe$ to \eqref{eq:Maxwell} satisfies the vector valued Helmholtz equation and since the Silver-M¸\"uller radiation condition is equivalent to the Sommerfeld radiation condition for the Cartesian components of $\bfe$ (see \cite{Colton}), we can extend the pole condition to exterior Maxwell problems using it for each component of $\bfe=(\fe_1,\fe_2,\fe_3)^T$. Hence, with the transformation of $H(\curl;\Omega)$ functions of Lemma~\ref{Lem:TrafoBasis} we define
\begin{equation}
  \label{eq:MapHcurl}
  \bFe(\bullet,\hat{x}):=\hat{J}^{T}(\hat{x}) \vector{\MT\LT
    \fe_1\circ \BLM(\bullet,\hat{x}) }{\Dop \MT \LT \fe_2\circ
    \BLM(\bullet,\hat{x})}{\Dop \MT \LT \fe_3\circ
    \BLM(\bullet,\hat{x})},\qquad \hat{x} \in T.
\end{equation}
The transformation of the mass integral in $H(\curl,\El)$ has already been used for gradient fields
and is given in analogy to \eqref{eq:stiffH1} by
\begin{equation}
\label{eq:massHcurl}
\int_\El \bfe \cdot \bfz dx =A_{G,T} \left( ((\Dop \otimes \id) \Fe_1,\Fe_2,\Fe_3)^T,((\Dop \otimes \id) \Fz_1,\Fz_2,\Fz_3)^T \right).
\end{equation}
Due to the transformation of the $\curl$ operator in  Lemma~\ref{Lem:TrafoBasis} we define $C=(c_{jk})_{j,k=1}^3:=|\hat{J}|^{-1} \hat{J}^T\hat{J}$ with 
\begin{equation}
  \label{eq:DefC}
  \frac{1}{|J|} J^T J  (\xi,\hat{x}) =\left(\!\begin{array}{ccc}
      \frac{c_{11}(\hat{x})}{(1+\xi)^2} &\frac{c_{12}(\hat{x})}{1+\xi}&\frac{c_{13}(\hat{x})}{1+\xi}\\
      \frac{c_{21}(\hat{x})}{1+\xi} &c_{22}(\hat{x})&c_{23}(\hat{x})\\
      \frac{c_{31}(\hat{x})}{1+\xi} &c_{32}(\hat{x})&c_{33}(\hat{x})\\
    \end{array}\!\right).
\end{equation}
The factor $(\xi+1)^{-1}$ in~\eqref{eq:DefC} leads to the integral
operator $\Iop:=\Dop^{-1}$ in the Hardy space formulation of the $\curl \curl$ integral:
\begin{equation}
\begin{aligned}
\label{eq:stiffHcurl}
&\int_\El \nabla \times \bfe \cdot \nabla \times \bfz ~dx =\\
& A_{C,T} \left(\vector{\left(\Iop \otimes \nabla_{\hat{x}}^{(1)}\right) \Fe_3 - \left(\Iop \otimes \nabla_{\hat{x}}^{(2)}\right) \Fe_2}{\left(\id \otimes \nabla_{\hat{x}}^{(2)}\right) \Fe_1 - \left(\hat \partial_\xi \otimes \id\right) \Fe_3}{\left(\hat \partial_\xi \otimes \id\right) \Fe_2-\left(\id \otimes \nabla_{\hat{x}}^{(1)}\right) \Fe_1} ,  \vector{\left(\Iop \otimes \nabla_{\hat{x}}^{(1)}\right) \Fz_3 - \left(\Iop \otimes \nabla_{\hat{x}}^{(2)}\right) \Fz_2}{\left(\id \otimes \nabla_{\hat{x}}^{(2)}\right) \Fz_1 - \left(\hat \partial_\xi \otimes \id\right) \Fz_3}{\left(\hat \partial_\xi \otimes \id\right) \Fz_2-\left(\id \otimes \nabla_{\hat{x}}^{(1)}\right) \Fz_1}\right).
\end{aligned}
\end{equation}
%
%
As in the
previous section, the integrals in~\eqref{eq:massHcurl} 
and~\eqref{eq:stiffHcurl} form the exterior part of~\eqref{eq:Maxwell}:
\begin{equation}
\begin{aligned}
&\LiF{\bfz} = \int_{\Oi}  \nabla \times \bfe \cdot \nabla \times \bfz - \ep \kappa^2\, \bfe\cdot \bfz dx + \sum_{T \in \mathcal{T}} \Big(\\
&A_{C,T} \left(\vector{\left(\Iop \otimes \nabla_{\hat{x}}^{(1)}\right) \Fe_3 - \left(\Iop \otimes \nabla_{\hat{x}}^{(2)}\right) \Fe_2}{\left(\id \otimes \nabla_{\hat{x}}^{(2)}\right) \Fe_1 - \left(\hat \partial_\xi \otimes \id\right) \Fe_3}{\left(\hat \partial_\xi \otimes \id\right) \Fe_2-\left(\id \otimes \nabla_{\hat{x}}^{(1)}\right) \Fe_1} ,  \vector{\left(\Iop \otimes \nabla_{\hat{x}}^{(1)}\right) \Fz_3 - \left(\Iop \otimes \nabla_{\hat{x}}^{(2)}\right) \Fz_2}{\left(\id \otimes \nabla_{\hat{x}}^{(2)}\right) \Fz_1 - \left(\hat \partial_\xi \otimes \id\right) \Fz_3}{\left(\hat \partial_\xi \otimes \id\right) \Fz_2-\left(\id \otimes \nabla_{\hat{x}}^{(1)}\right) \Fz_1}\right)\\
&- A_{\ep G,T} \left( ((\Dop \otimes \id) \Fe_1,\Fe_2,\Fe_3)^T,((\Dop \otimes \id) \Fz_1,\Fz_2,\Fz_3)^T \right) \Big).
\end{aligned}
\end{equation}

\subsection{Exterior variational formulation in $H(\divo;\Oe)$}
\label{sec:HSMHdiv}
The construction follows analogously to the $H(\curl;\Omega)$ functions of
the previous section using the transformation of $H(\divo;\Omega)$ function
of Lemma~\ref{Lem:TrafoBasis}:
\begin{equation}
  \label{eq:MapHdiv}
  \bFe(\bullet,\hat{x}):=|\hat{J}(\hat{x})|\hat{J}^{-1}(\hat{x})  
  \vector{ \Dop \Dop \MT\LT  ~\fe_1\circ \BLM(\bullet,\hat{x}) }{
    \Dop \MT \LT ~\fe_2\circ \BLM(\bullet,\hat{x})}{
    \Dop \MT \LT ~\fe_3\circ \BLM(\bullet,\hat{x})},\qquad \hat{x} \in T,
\end{equation}
for $\bfe=(\fe_1,\fe_2,\fe_3)^T \in H(\divo,\El)$. Hence, the mass
integral in $H(\divo,\El)$ is the integral in~\eqref{eq:stiffHcurl}. The stiffness integral
follows by straightforward calculations using the transformation of
the divergence operator.

\section{Local exact sequence of tensor product spaces}
\label{sec:BasisFunc}
For a bounded domain $\Omega \subset \setR^3$, which is diffeomorphic to
the unit ball the sequence 
\begin{equation}
  \label{eq:exsec}
  \setC \ \RA{j} \ H^1(\Omega)\ \RA{\nabla} \ H(\curl,\Omega)  \ 
  \RA{\nabla \times} \ H(\divo,\Omega) \ \RA{\nabla \cdot} \ L^2(\Omega) 
  \ \RA{\Null} \ \{0\}
\end{equation}
is exact, i.e.\ $j \setC = \ker(\nabla)$ (where $j$ maps $a\in\setC$
to the constant function $\Omega\to \setC$, $x\mapsto a$), 
$\nabla H^1(\Omega) = \ker(\nabla \times)$,
$\nabla \times H(\curl,\Omega) = \ker (\nabla \cdot)$ and
$\nabla \cdot H(\divo,\Omega) = L^2(\Omega)$.
On the two dimensional surface $\Gamma$ the sequence~\eqref{eq:exsec} 
decomposes into two exact sequences
\begin{subequations}
  \label{eq:exsec2d}
  \begin{eqnarray}
    &&\setC \ \RA{j_{\hat{x}}} \ H^1(\Omega) \ \RA{\nabla_{\hat{x}}} \ 
    H(\scurl,\Omega) \ \RA{\nabla_{\hat{x}}^\perp \cdot} \ L^2(\Omega) 
    \ \RA{\Null} \ \{0\}, \label{eq:exsec2da}
    \\
    &&\setC \ \RA{j_{\hat{x}}} \ H^1(\Omega) \ \RA{\nabla_{\hat{x}}^\perp} \
    H(\divo,\Omega) \ \RA{\nabla_{\hat{x}}\cdot} \ L^2(\Omega) \ \RA{\Null} 
    \ \{0\}, \label{eq:exsec2db}
  \end{eqnarray}
\end{subequations}
with the rotated surface gradient operator
$\nabla_{\hat{x}}^\perp:=(-\nabla_{\hat{x}}^{(2)},\nabla_{\hat{x}}^{(1)})^T$
and the scalar curl operator $\scurl_{\hat{x}}=\nabla_{\hat{x}}^\perp
\cdot\,$. In the following $\bullet^\perp$ always denotes the rotation of a 2d vector, vector space or operator with values in $\setC^2$.

For Maxwell eigenvalue problems, due to exact sequence property there
exists an eigenvalue $0$ with an infinite dimensional eigenspace consisting
of all gradients of $H^1$ functions. If
the finite element method is not constructed carefully, the numerical
results will be polluted by artificial eigenvalues coming from this
eigenspace (see e.g.~\cite[Sec. 6.2]{Boffietal:99}). Hence, 
following~\cite[Chapt. 5]{Monk:03} or~\cite{Zaglmayr:06} we try to carry over
the exact sequence property of the continuous spaces to the discrete
spaces. Until now we cannot prove a discrete de Rham diagram with commuting
projection and differential operators for the discrete spaces we are
going to construct, but at least they build an exact sequence locally.

In the interior domain we use a discretization with tetrahedral elements 
and the high order elements of~\cite[Chapt. 5.2.6]{Zaglmayr:06} with the 
local exact sequence
\begin{equation}
  \setC \ \RA{j} \ P^p\ \RA{\nabla} \ \left(P^{p-1}\right)^3  \ 
  \RA{\nabla \times} \ \left(P^{p-2}\right)^3 \ \RA{\nabla \cdot} \ 
  \left(P^{p-3}\right)^3  \ \RA{\Null} \ \{0\}.
\end{equation}
In the following we focus on the elements for the exterior domain and the 
coupling on the artificial boundary.

There, the pyramidal frustums $\El$ are mapped into right prisms $\hat
\El$ (see Fig.~\ref{fig:segTrafo}). Therefore the local elements are
build by tensor products of one-dimensional infinite elements for $\xi
\in [0,\infty)$ 
and triangular elements on the 2d surface element $T$ consisting of subspaces $W_T
\subset H^1(T)$, $V_T \subset H(\scurl,T)$ and $X_T \subset L^2(T)$. 
The construction of the local
sequence is motivated by a modified tensor product of two chain complexes
\cite[Sec. 5.7]{HiltonWylie:60}:
Combining the two surface sequences \eqref{eq:exsec2d} we use the cochain complex array 
\begin{equation}
\label{eq:ChainDia}
\begin{CD}
W_\xi \otimes W_T
@>\id \otimes \nabla_{\hat{x}}>> 
W_\xi \otimes V_T
@>\id \otimes  \nabla_{\hat{x}}^\perp \cdot >>
W_\xi \otimes X_T \\ 
@V{\partial_\xi\otimes\id}VV  
@V{\partial_\xi\otimes \id^\perp}VV 
@V{\partial_\xi\otimes\id}VV \\ 
W_\xi' \otimes W_T
@>\id\otimes \nabla_{\hat{x}}^\perp>>  
 W_\xi' \otimes V_T^\perp
@>\id\otimes \nabla_{\hat{x}}\cdot>> 
W_\xi' \otimes X_T.
\end{CD}
\end{equation}
The tensor product chain complex is given by the sums over the diagonals:
\begin{equation}
   \label{eq:TensorChain}
\begin{CD}
W_\xi \otimes W_T
@>{\smat{\partial_\xi \otimes \id \\ \id \otimes \nabla_{\hat{x}}}}>>
\vect{W_\xi' \otimes W_T}{W_\xi \otimes V_T}
@>{\smat{0 & \id \otimes  \nabla_{\hat{x}}^\perp \cdot \\ - \id \otimes \nabla_{\hat{x}}^\perp & \partial_\xi \otimes \id^\perp}}>> \cdots \\
... @>>> \vect{W_\xi \otimes X_T}{W_\xi' \otimes V_T^\perp}
@>{\smat{\partial_\xi \otimes \id & \id \otimes \nabla_{\hat{x}} \cdot}}>>
W_\xi' \otimes X_T.
\end{CD}
\end{equation}

If $\Gamma \subset \setR^2$, the operators in ~\eqref{eq:TensorChain} are the standard gradient, $\curl$ and divergence operator. Hence, ~\eqref{eq:TensorChain} discretizes in some sense
\begin{equation}
 H^1(\El)\ \RA{\nabla} \ H(\curl,\El)  \ \RA{\nabla \times} \ H(\divo,\El) \ 
 \RA{\nabla \cdot} \ L^2(\El).
\end{equation}
The term ,,in some sense'' is used, because the finite element basis
function will not belong to $H^1(\El)$, $H(\curl,\El)$, $
H(\divo,\El)$ and $L^2(\El)$, since the radial parts of them belong to
the Hardy space $H^+(S^1)$. The definition of the correct function
spaces would be very complicated. In~\cite[Sec. 3.2]{HohageNannen:09}
the function space is given for scalar functions with spherical
artificial boundary. Here, we confine ourselves to the discrete
spaces, but we indicate in the following with the notation
\begin{equation}
 \hat \nabla := \left(\hat \partial_\xi,\nabla_{\hat{x}}^{(1)},\nabla_{\hat{x}}^{(2)}\right)^T,
\end{equation}
that $\hat \partial_\xi$ is an operator of a subset of $H^+(S^1)$ into $H^+(S^1)$.

\subsection{$H^1$-element}
\label{sec:H1element}
We briefly discuss the $H^1$-elements proposed in~\cite{HohageNannen:09,NannenSchaedle:09}. 
For each segment $\El$ we use the local tensor product space
\begin{equation}
\label{eq:H1space}
W_{\hat{\El}}= W_\xi \otimes W_T,
\end{equation}
whereas $W_T$ is the surface element of the surface triangle $T$ of
the $H^1$ tetrahedral volume element and $W_\xi$ a space in radial
direction. $W_T$ equals the usual polynomial space $P^p(T)$ for the
two-dimensional triangular element and  therefore it has the dimension
$\frac{1}{2}(p+2)(p+1)$.

In radial direction we have to discretize the Hardy space $H^+(S^1)$, where a $L^2$-orthogonal basis is given by monomials. Using the first
$N+1$ monomials $\Pi_N:={\rm \bf span\/} \{z^0,...,z^N\}$ and the operator $\OpT_-$ of \eqref{eq:DefOpT}, we define the discrete space $W_\xi:=\frac{1}{i \kappa_0} \OpT_- (\setC
\times \Pi_N)$ with the basis functions
\begin{equation}
  \Psi_{-1}:=\frac{1}{i \kappa_0} \OpT_- (1,0) ~\text{and}~\Psi_j:=
  \frac{1}{i \kappa_0} \OpT_- (0,(\bullet)^j),~j=0,...,N.
\end{equation}
Since $\left(\LT^{-1} \MTinv \Psi_{j} \right)(0)=0$ for $j=0,...,N$ and $\left(\LT^{-1} \MTinv \Psi_{-1} \right)(0)=1$, the basis function $\Psi_{-1}$ is used to couple the Hardy space infinite elements of the exterior domain $\Oe$ to the finite elements of the interior domain $\Oi$.

\begin{figure}
\centering
\resizebox{0.6\textwidth}{!}{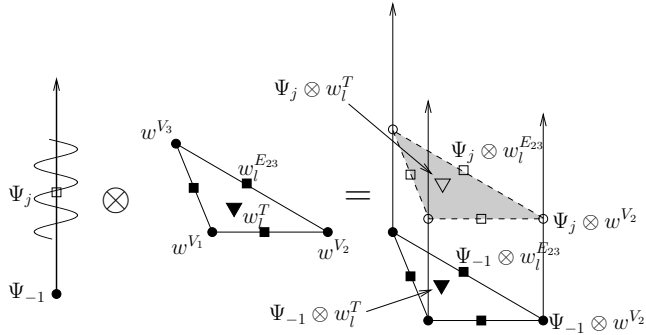}
\caption{basis functions in $W_{\hat{\El}}= W_\xi \otimes W_T$}
\label{fig:basisH1}
\end{figure}

Let us collect all basis functions for the prism $\hat{\El}$ and assign them to a vertex $V_i$ ($i=1,2,3$)
of the surface triangle $T$, an edge $E_{ij}$ between vertex $V_i$ and
$V_j$, the surface triangle $T$, an infinite ray $R_i$ corresponding
to a vertex $V_i$, an infinite face $F_{ij}$ corresponding to the edge
$E_{ij}$ or the prism $\hat{\El}$ itself. Following~\cite[Chapt. 5.2.3]{Zaglmayr:06} 
a basis function $w$ of $W_{T}$ can be a
\begin{enumerate} \itemsep0ex
 \item vertex basis function $w^{V_i}$,
 \item edge basis function $w^{E_{ij}}_l$, $l=1,...,p-1$, or a 
 \item element basis function $w^{T}_l$, $l=1,...,\frac{1}{2}(p-2)(p-1)$.
\end{enumerate}
Hence, the basis functions of $W_{\hat{\El}}=W_\xi \otimes W_{T}$ sketched in
Fig. \ref{fig:basisH1} are
\begin{enumerate}\itemsep0ex
 \item vertex functions $W^{V_i}:=\Psi_{-1} \otimes w^{V_i}$,
 \item edge functions $W^{E_{ij}}_l:=\Psi_{-1} \otimes w^{E_{ij}}_l$, $l=1,...,p-1$, 
 \item surface functions $W^{T}_l:=\Psi_{-1} \otimes w^{T}_l$, $l=1,...,\frac{1}{2}(p-2)(p-1)$,
 \item ray functions $W_k^{R_i}:=\Psi_k \otimes w^{V_i}$, $k=0,...,N$,
 \item infinite face functions $W_{k,l}^{F_{ij}}:=\Psi_k \otimes w^{E_{ij}}_l$, $k=0,...,N$, $l=1,...,p-1$ and
 \item segment functions $W_{k,l}^\El:=\Psi_k \otimes w^{T}_l$, $k=0,...,N$, $l=1,...,\frac{1}{2}(p-2)(p-1)$.
\end{enumerate}
The degrees of freedom for the segment basis functions are interior
degrees of freedom, while the others have to coincide between
neighboring infinite segments and on the surface with the tetrahedron
in $\Oi$. In total, there are
\begin{equation}
\dim W_{\hat{\El}}= \frac{1}{2}(p+2)(p+1) (N+2)
\end{equation}
degrees of freedom. 

\subsection{$H(\curl)$-element}
Due to the sequence~\eqref{eq:TensorChain} we 
use for the approximation of elements in $H(\curl)$ the space
\begin{equation}
 V_{\hat{\El}} = \vect{W_\xi' \otimes W_T}{W_\xi \otimes V_T},
\end{equation}
with the surface finite element space $V_T$ of the $H(\curl)$ volume element of
order $p-1$. As in the previous case $V_T$ is the 
vectorial triangular element space with two components, 
i.e. $V_T=\left(P^{p-1}\right)^2$ with $ \dim V_T
= (p+1)p$. For $W_\xi'$ ~\eqref{eq:derxihat} leads to the basis functions
\begin{equation}
 \psi_{-1}:=\OpT_+ (1,0) ~\text{and}~\psi_j:=\OpT_+ (0,(\bullet)^j),~j=0,...,N.
\end{equation}

\begin{figure}
\centering
\subfigure[$W_\xi \otimes V_T$]{\resizebox{0.45\textwidth}{!}{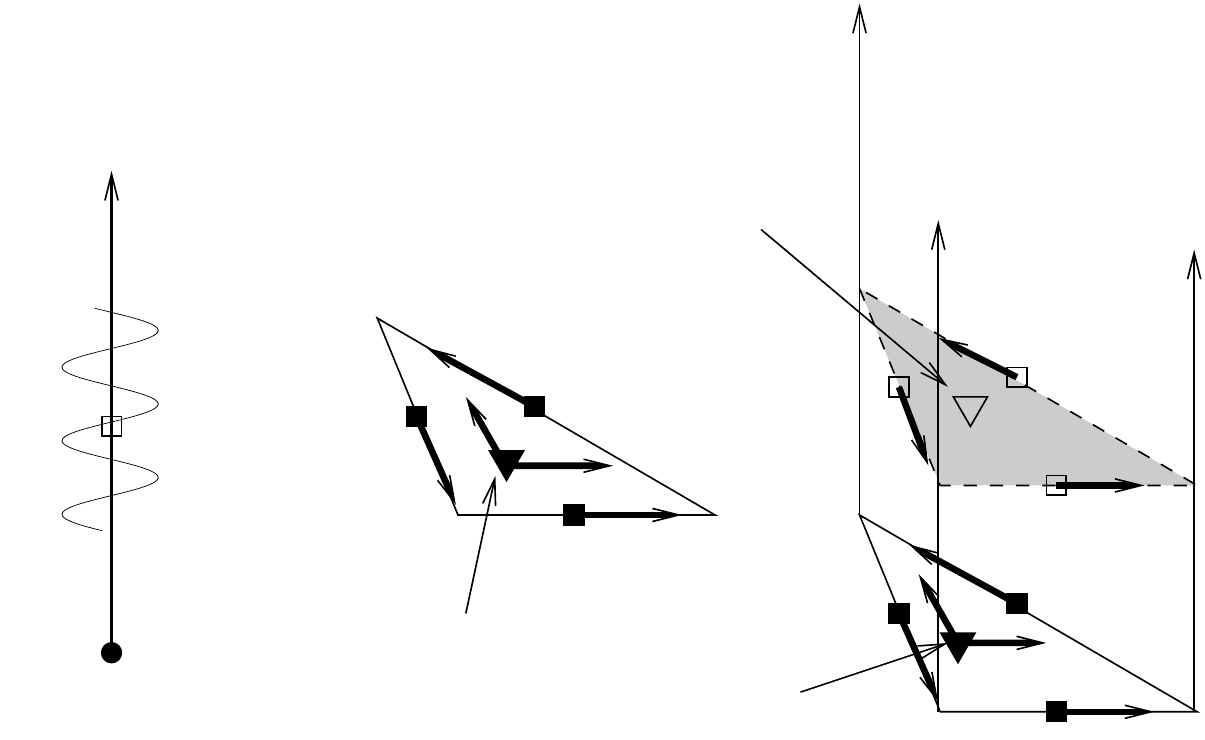}}\hfill
\subfigure[$W_\xi' \otimes W_T$]{\resizebox{0.45\textwidth}{!}{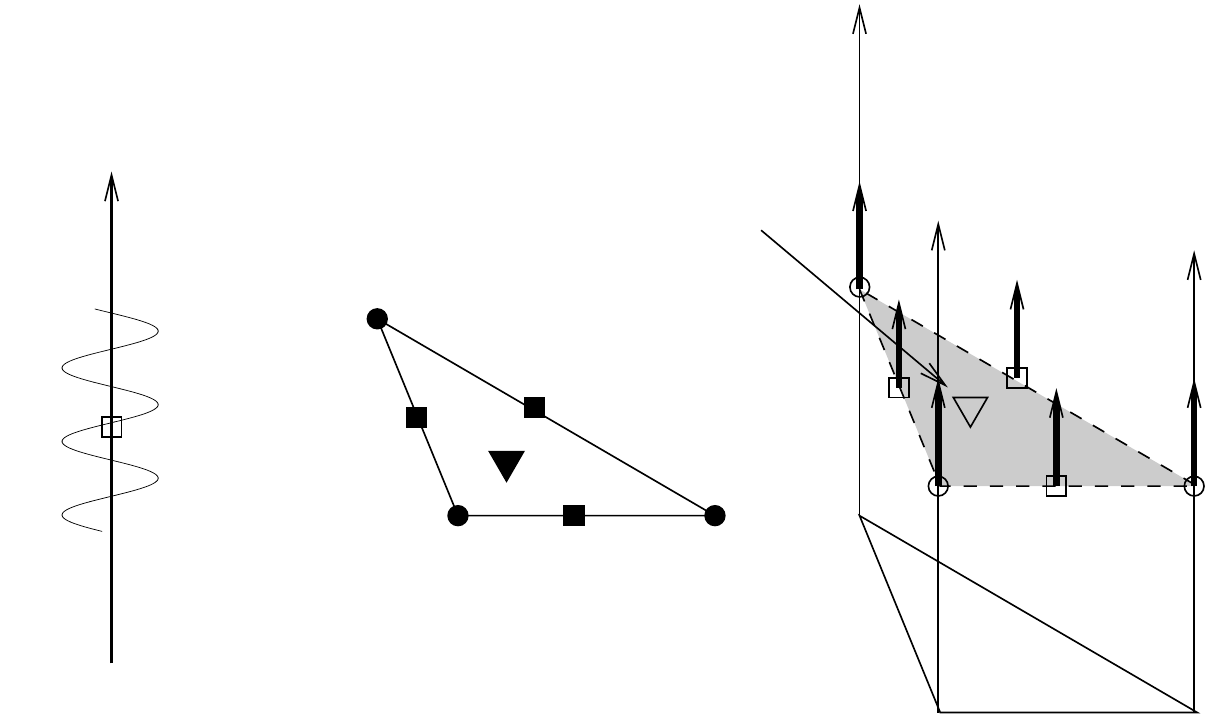}}
\caption{basis functions of $V_{\hat{\El}}$}
\label{fig:basisHcurl}
\end{figure}

If $\mathbf{v}^{E_{ij}}_l \in V_T$ denotes the edge basis functions of the
$H(\curl)$ surface triangular element and $\mathbf{v}^{T}_l$ the surface
triangle basis functions, the basis functions in $V_{\hat{\El}}$ are 
\begin{enumerate} \itemsep0ex
  \item edge functions $\mathbf{V}_l^{E_{ij}}:=\vect{0}{\Psi_{-1} \otimes \mathbf{v}^{E_{ij}}_l}$, $l=1,...,p$,
  \item surface functions $\mathbf{V}_l^T:=\vect{0}{\Psi_{-1} \otimes \mathbf{v}^{T}_l}$, $l=1,...,(p-2)p$,
  \item ray functions $\mathbf{V}_k^{R_i}:=\vect{\psi_k \otimes w^{V_i}}{\boldsymbol{0}}$, $k=-1,...,N$,
  \item two types of infinite face functions 
\begin{enumerate}
 \item $\mathbf{V}_{k,l}^{F_{ij},1}:=\vect{0}{\Psi_{k} \otimes \mathbf{v}^{E_{ij}}_l}$, $k=0,...,N$, $l=1,...,p$,
 \item $\mathbf{V}_{k,l}^{F_{ij},2}:=\vect{\psi_k \otimes w^{E_{ij}}_l}{\boldsymbol{0}}$, $k=-1,...,N$, $l=1,...,p-1$
\end{enumerate}
  \item and two types of segment functions
\begin{enumerate}\itemsep0ex
 \item $\mathbf{V}_{k,l}^{\El,1}:=\vect{0}{\Psi_k \otimes \mathbf{v}^{T}_l}$, $k=0,...,N$, $l=1,...,(p-2)p$ and
 \item $\mathbf{V}_{k,l}^{\El,2}:=\vect{\psi_k \otimes w^{T}_l}{\boldsymbol{0}}$. $k=-1,...,N$, $l=1,...,\frac{1}{2}(p-2)(p-1)$.
\end{enumerate}
\end{enumerate}

See Fig.~\ref{fig:basisHcurl} for a scheme of the curl-conforming basis functions. 
Note, that only the tangential directions indicated by the arrows are continuous over the boundaries. The
segment basis functions have vanishing tangential components on the
edges, rays, faces and the surface and are therefore interior basis
functions. The tangential components of the other functions vanish on
the boundary parts to which they are not assigned. For example the
tangential component $\psi_k \otimes w^{V_1}$ of a ray function on
$R_1$ vanishes due to the vertex function $w^{V_1}$ on the rays $R_2$
and $R_3$ as well as on the infinite face $F_{23}$. It does not vanish
on the surface, but there it is the normal component, which does not
need to be continuous for $H(\curl)$ functions. The tangential
component on the surface is the second component of the vector and
therefore zero on the whole segment.

If we link together the degrees of freedom for neighboring elements, we get 
a $H(\curl)$-conforming method with 
\begin{equation}
\label{eq:DimHcurl}
 \dim V_{\hat{\El}} = (N+2) \left( \frac{1}{2}(p+2)(p+1) ~+~(p+1)p\right).
\end{equation}
degrees of freedom for each segment.

\begin{Rem}
It is not essential to use for $W_\xi'$ the basis functions
$\psi_k$. The special choice of the basis functions $\Psi_k$ was necessary to divide into
boundary degrees of freedom and interior ones, which is useful in
order to guarantee continuity of the global basis functions at the
segment boundaries and hence to get a conforming method. But as we
have seen for the ray function on $R_1$, there is no coupling between
the ray and face basis functions and the basis functions in the
tetrahedron. 

Moreover, all of the functions $\LT^{-1} \MTinv \psi_k$
have non-vanishing boundary values. Since $\OpT_+ (\setC,\Pi_N)^T=\Pi_{N+1}$
we could also directly take $\{z^0,...,z^{N+1}\}$ as basis
functions. The only reason for us using $\psi_k$ is, that in this way
the basis functions in $W_\xi'$ are exactly the derivatives of the
basis functions in $W_\xi$.
\end{Rem}

\subsection{$H(\divo)$-element}
\label{sec:Hdivelem}
Starting from $V_{\hat{\El}}$, we calculate
\begin{equation}
 \hat \nabla \times V_{\hat{\El}} = 
 \left(\begin{array}{cc} 0 & \id \otimes \nabla_{\hat{x}}^\perp \cdot \\ -\id \otimes \nabla_{\hat{x}}^\perp & \hat \partial_\xi \otimes \id^\perp  \end{array} \right) \vect{W_\xi' \otimes W_T}{W_\xi \otimes V_T}.
\end{equation}
The space $H(\divo_{\hat{x}})$ is due 
to~\eqref{eq:exsec2d} a rotated $H(\scurl_{\hat{x}})$
and $\nabla_{\hat{x}}^\perp W_T \subset V_T^\perp$. Moreover,
the normal components of the basis functions on each face of a
$H(\divo)$ tetrahedral element include the scalar
$\scurl_{\hat{x}}$-fields of $V_T$. Therefore we chose
\begin{equation}
  Q_{\hat{\El}} = \vect{W_\xi \otimes Q_T} {W_\xi' \otimes V_T^\perp},
\end{equation}
where $Q_T=P^{p-2}$ is the finite element space for the normal component 
on the surface of a $H(\divo)$ volume element of order $(p-2)$ 
and takes the place of $X_T$ in the third term of~\eqref{eq:TensorChain}.

Let $q^{T}$ be a basis functions in $Q_T$. Then, the basis functions of 
$Q_{\hat{\El}}$ consist of
\begin{enumerate} \itemsep0ex
  \item surface functions 
    $\mathbf{Q}_l^T:=\vect{\Psi_{-1} \otimes q^{T}_l}{\boldsymbol{0}}$, 
    $l=1,...,\frac{1}{2}p(p-1)$,
  \item face functions 
    $\mathbf{Q}_{k,l}^{F_{ij}}:=\vect{0}{\psi_k \otimes \left(\mathbf{v}^{E_{ij}}_l\right)^\perp}$, 
    $k=-1,...,N$, $l=1,...,p$
  \item and two types of segment functions
\begin{enumerate}\itemsep0ex
 \item $\mathbf{Q}^{\El,1}_{k,l}:=\vect{\Psi_k \otimes q^{T}_l}{0}$, 
   $k=0,...,N$, $l=1,...,\frac{1}{2}p(p-1)$ and
 \item $\mathbf{Q}^{\El,1}_{k,l}:=\vect{0}{\psi_k \otimes \left( \mathbf{v}^{T}_l \right)^\perp}$, 
   $k=-1,...,N$, $l=1,...,(p-2)p$.
\end{enumerate}
\end{enumerate}
The basis functions of $W_\xi \otimes Q_T$ ensure the continuity of the normal component on the
boundary surface and those of $W_\xi' \otimes V_T^\perp$ the continuity of the normal component on the infinite faces.

For each segment there are
\begin{equation}
\label{eq:DimHdiv}
 \dim Q_{\hat{\El}} = (N+2) \left(\frac{1}{2} p (p-1) + (p+1)p \right).
\end{equation}
degrees of freedom.

\subsection{$L^2$-element}
\label{sec:Ltelem}
The $L^2$-element has no coupling on the boundaries. Nevertheless, we
compute the divergence of $Q_{\hat{\El}}$:
\begin{equation}
  \label{eq:divQS}
 \hat \nabla \cdot Q_{\hat{\El}} = \vect{\hat \partial_\xi \otimes \id}{ \id \otimes \nabla_{\hat{x}}} \cdot \vect{W_\xi \otimes Q_T} {W_\xi' \otimes V_T^\perp} = W_\xi' \otimes Q_T + W_\xi' \otimes \nabla_{\hat{x}} V_T^\perp.
\end{equation}
Both parts fit together and therefore we define
\begin{equation}
  X_{\hat{\El}} = W_\xi' \otimes X_T,
\end{equation}
with a two-dimensional triangular element $X_T=P^{p-2}$ of order
$p-2$. In the interior domain we use volume elements of order $p-3$,
but since $L^2$-elements do not require continuity, there is no
conflict.
\begin{equation}
  \label{eq:ndofsXS}
  \dim X_{\hat{\El}} = \frac{1}{2} p (p-1) (N+2) .
\end{equation}

\section{Properties of the sequence}
\label{sec:Sequence}
The analysis for this sequence is far from being complete. However, we
are able to prove the exact sequence property locally.

\subsection{Local properties}
By construction it holds for each infinite element
\begin{eqnarray*}
  \hat \nabla W_{\hat{\El}} &\subset& 
  \{\mathbf{v} \in V_{\hat{\El}} ~|~\hat \nabla \times \mathbf{v} =0 \},\\
  \hat \nabla \times V_{\hat{\El}} &\subset& 
  \{\mathbf{v} \in Q_{\hat{\El}} ~|~\hat \nabla \cdot \mathbf{v} =0 \},\\
  \hat \nabla \cdot Q_{\hat{\El}} &\subset& X_{\hat{\El}}.
\end{eqnarray*}
Note, that constant functions are not outgoing. 
Hence $\id \setC = \ker(\nabla)$ is not meaningful in this context. 
By counting the degrees of freedom we get the following theorem.
\begin{Theo}
  \label{theo:locexseq}
  The discrete tensor product spaces defined above build a local exact sequence, i.e.
  \begin{equation}
    W_\El\ \RA{\hat \nabla} \ V_\El  \ \RA{\hat \nabla \times} \ Q_\El \ \RA{\hat \nabla \cdot} \ X_\El \ \RA{\Null} \ \{0\}.
  \end{equation}
\end{Theo}

\begin{proof}
For a linear operator $T$ on a finite dimensional space $D$ we have
the identity $\dim D=\dim T(D) + \dim \ker(T)$. From \eqref{eq:divQS}
we conclude, that $\dim \hat \nabla \cdot( Q_{\hat{\El}}) \geq
\frac{1}{2}p(p-1)(N+2)$, since $\hat \partial_r W_r \otimes Q_t$ has
this dimension. With \eqref{eq:ndofsXS} we get $\hat \nabla \cdot(
Q_{\hat{\El}})= X_{\hat{\El}}$ and with \eqref{eq:DimHdiv} $\dim
\ker(\hat \nabla \cdot)= (N+2)(p+1)p $. Again, since $\dim \hat\nabla
\times V_{\hat{\El}}\geq (N+2)(p+1)p$, we deduce $\hat\nabla \times
V_{\hat{\El}}=\ker(\hat \nabla \cdot)$ and using \eqref{eq:DimHcurl}
$\dim \ker(\hat \nabla \times)= \frac{1}{2}(p+2)(p+1)(N+2)$ . Last,
$\dim \hat \nabla W_{\hat{\El}}\geq \frac{1}{2}(p+2)(p+1)(N+2)$ and
hence $\hat \nabla W_{\hat{\El}}=\ker(\hat \nabla \times)$ and
$\ker(\hat \nabla)=\{0\}$.
\end{proof}

\subsection{Global properties}
For the global finite element spaces in the exterior domain we define the spaces for the different kinds of degrees of freedom
\begin{subequations}
\label{eq:DefWext}
\begin{eqnarray}
 W_{\rm ext}^{\rm V} &:=& \bigcup_{V } ~\{ W^{V} \},\\
 W_{\rm ext}^{\rm E} &:=& \bigcup_{E } ~\{ W_l^E ~|~l=1,...,p-1\},\\
 W_{\rm ext}^{\rm T} &:=& \bigcup_{T } ~\{ W_l^T ~|~l=1,...,\frac{1}{2}(p-2)(p-1)\},\\
 W_{\rm ext}^{\rm R} &:=& \bigcup_{R } ~\{ W_k^R ~|~k=0,...,N\},\\
 W_{\rm ext}^{\rm F} &:=& \bigcup_{F } ~\{ W_{k,l}^F ~|~k=0,...,N, l=1,...,p-1\},\\
 W_{\rm ext}^{\rm \El} &:=& \bigcup_{\El } ~\{ W_{k,l}^\El ~|~k=0,...,N, l=1,...,\frac{1}{2}(p-2)(p-1)
\end{eqnarray}
\end{subequations}
and finally collect them together to
\begin{equation}
 W_{\rm ext}:=W_{\rm ext}^{\rm V} \cup W_{\rm ext}^{\rm E} \cup W_{\rm ext}^{\rm T} \cup W_{\rm ext}^{\rm R} \cup W_{\rm ext}^{\rm F} \cup W_{\rm ext}^{\rm \El}.
\end{equation}
The spaces $V_{\rm ext}$, $Q_{\rm ext}$ and $X_{\rm ext}$ are defined analogously. In this way theorem \ref{theo:locexseq} should carry over into the global finite element spaces.

\section{Numerical results}
\label{sec:Numerics}
Although there is no convergence result of the Hardy space method for
electromagnetic scattering or resonance problems, we can expect from
the results for scalar problems (see~\cite{HohageNannen:09} and the
convergence plots in~\cite{NannenSchaedle:09}) super algebraic
convergence of the method w.r.t. the number of degrees of freedom in
the Hardy space. In order to substantiate this predication, we give in
the following some numerical tests for scattering as well as resonance
problems. Especially the results of the challenging problems in
Sections~\ref{sec:MagnDipole}, \ref{sec:ResPyramid} and \ref{sec:hdiv}
indicate, that only a few basis functions in the Hardy space are
needed.

For all computations the mesh generator netgen~\cite{netgen} together
with the high order finite element code ngsolve is used for the
interior problem. Using a source term $\LiF{\fz}\neq0$
in~\eqref{eq:Maxwell} or~\eqref{eq:Helmholtz} and a given wavenumber
$\kappa>0$ leads to a system of linear equations, which we solve
directly with the package PARDISO~\cite{Pardiso}. A resonance problem
consist of finding $\kappa$ with $\Re(\kappa)>0$ and non-trivial
resonance functions $\fe$ solving the homogeneous (no sources,
vanishing boundary conditions) problems \eqref{eq:Maxwell},
\eqref{eq:Helmholtz} or \eqref{eq:Helmholtzdiv} and results in a
generalized matrix eigenvalue problem of the form
\begin{equation*} S \fe_h = \kappa_h^2 M u_h
\end{equation*} with symmetric, non-hermitian complex matrices $S$ and
$M$. This problem is solved using a shifted Arnoldi algorithm.

\subsection{Resonances of a Helmholtz resonator}

\begin{figure}
 \centering
\resizebox{0.8\textwidth}{!}{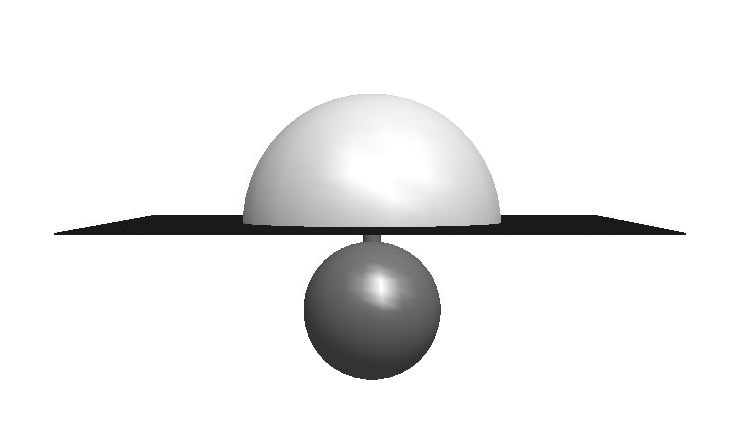}
\caption{Helmholtz resonator}
\label{fig:HelmResPlot}
\end{figure}

In the first example we compute resonances of a three dimensional
acoustic Helmholtz resonator and show the system response to a source
given by an incident plane wave $u_i(x)=e^{i\kappa d\cdot x}$ with 
$d=(-1/ \sqrt{2},0,-1/ \sqrt{2})^T$ and fixed wave number
$\kappa>0$. A single resonator consists of a ball with radius $0.062$, which
is connected via a cylinder with radius $0.0098$ and length $0.5$ to the
half space $\{x \in \setR^3:~x_3\geq 0\}$, cf. Fig. \ref{fig:HelmResPlot}. 
The infinite half space is truncated
using a half ball in the case of the scattering problem and a cube in
the case of the resonance problem.

We assume sound-hard boundary conditions at the boundary of the
resonator as well as at the infinite plane $\{x \in \setR^3:~x_3=
0\}$. For the Hardy space method no sources outside the finite element
domain $\Oi$ are allowed and the solution of the problem needs to be
outgoing in $\Oe$. Therefore, in $\Oi$ we compute the total field $u$
and in $\Oe$ the field $u_s=u-u_i-u_r$ with $u_r$ being the reflected
wave of the unperturbed half space $u_r(x)=e^{i\kappa \tilde d \cdot
x}$ with $\tilde d=(-1/ \sqrt{2},0,1/ \sqrt{2})^T$. $u_s$ is
outgoing with vanishing Neumann boundary values at the infinite plane,
since it is the perturbation of $u_i+u_r$ caused by the
resonator. Using $u$ in $\Oi$ and $u_s$ in $\Oe$ leads to a jump
condition at the artificial boundary and an additional boundary term
for the jump in the Neumann values.

\begin{figure}
  \centering
 \subfigure[$\kappa\approx 2.121- 3.36 \cdot 10^{-3}i$]{
   \includegraphics[width=0.25\textwidth]{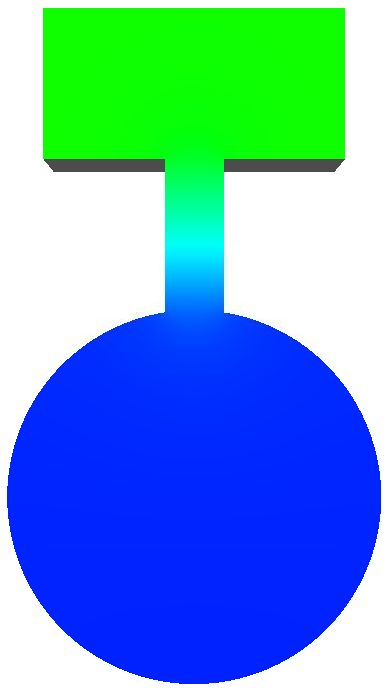}}\hfill
 \subfigure[$\kappa\approx 33.34- 3.66 \cdot 10^{-2}i$]{
   \includegraphics[width=0.25\textwidth]{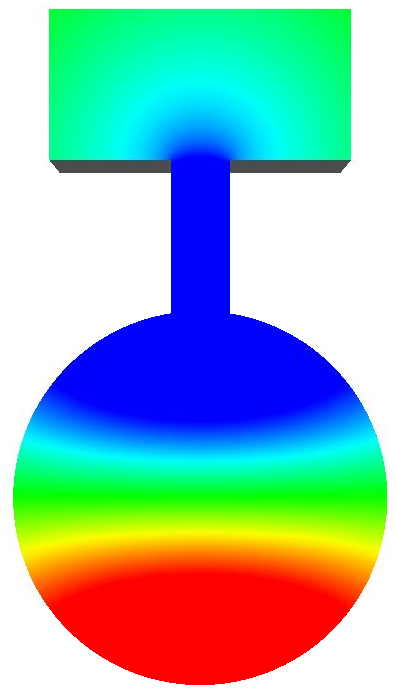}}\hfill
 \subfigure[$\kappa\approx 33.59- 6.85  10^{-11}i$]{
   \includegraphics[width=0.25\textwidth]{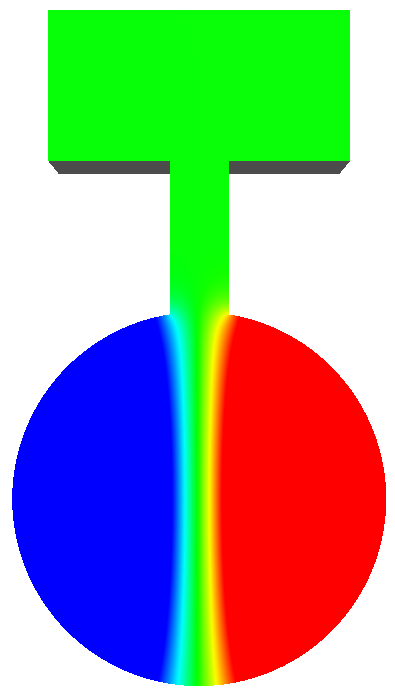}}
 \caption{Real part of three resonance functions and corresponding resonances of a single Helmholtz 
   resonator with volume $1 \mathrm{l}$, 
   length of the neck $5\mathrm{cm}$ and cross sectional area of the neck $3\mathrm{cm}^2$.}
 \label{fig:SingleHelmResonator}
\end{figure}

Fig.~\ref{fig:SingleHelmResonator} gives cross-sections of a single
Helmholtz resonator with three different resonance functions and the
corresponding resonances. For $c=343 m/s$ we calculate the
frequency $f=c \Re(\kappa)/2 \pi \approx 115.8 Hz $,
of the most relevant first resonance which has a quality factor
$Q=\Re(\kappa)/|\Im(\kappa)|\approx 631$. The approximative formula
(see e.g. \cite{Schroederetal:07}) for a Helmholtz resonator with
volume $V$, neck length $L$ and cross sectional area of the neck $S$ gives
$ f=\frac{c}{2 \pi} \sqrt{\frac{S}{L V}} \approx 133 Hz$.

For these computations a mesh of $2477$ tetrahedrons with finite
element order $7$, the number of degrees of freedom in radial
direction $N=15$ and $\kappa_0=8$ lead to approximately $155000$
degrees of freedom. Note, that the third resonance in
Fig.~\ref{fig:SingleHelmResonator} has multiplicity $2$.

\begin{figure}
  \centering
 \includegraphics[width=0.45\textwidth]{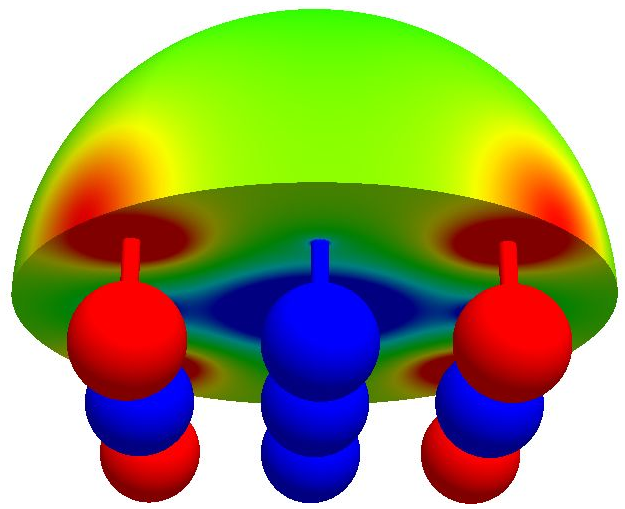}
 \caption{Real part of the total field to an incident plane wave 
   $u_i(x)=e^{i\kappa d\cdot x}$ with $d=(-1/ \sqrt{2},0,-1/ \sqrt{2})^T$ 
   and $\kappa=2.1241$. 
   The values of the total field vary between $\pm 2000$, 
   but the figure is limited to values between $\pm 10$.}
 \label{fig:HelmResonator}
\end{figure}

Fig.~\ref{fig:HelmResonator} shows the real part of the total field to
an incident plane wave with wavenumber $\kappa=2.1241$ for $9$ uniformly
distributed resonators. For this system of resonators the first $9$
resonance frequencies are between $114.8 Hz$ ($\kappa\approx 2.103$)
and $116.1 Hz$ ($\kappa\approx2.126$) with quality factors between
from $162$ to $864765$. For the resonance computations of the system of
resonators we used a mesh with $24394$ tetrahedrons, finite
element order $7$, $N=15$ and $\kappa_0=4$.

\subsection{Convergence test of a magnetic dipole}
\label{sec:MagnDipole}
\begin{figure}
 \centering
\subfigure{\includegraphics[width=0.3\textwidth]{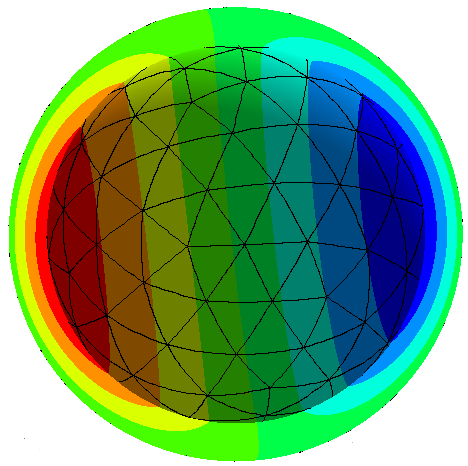}} \hspace{0.1\textwidth}
\subfigure{\includegraphics[width=0.45\textwidth]{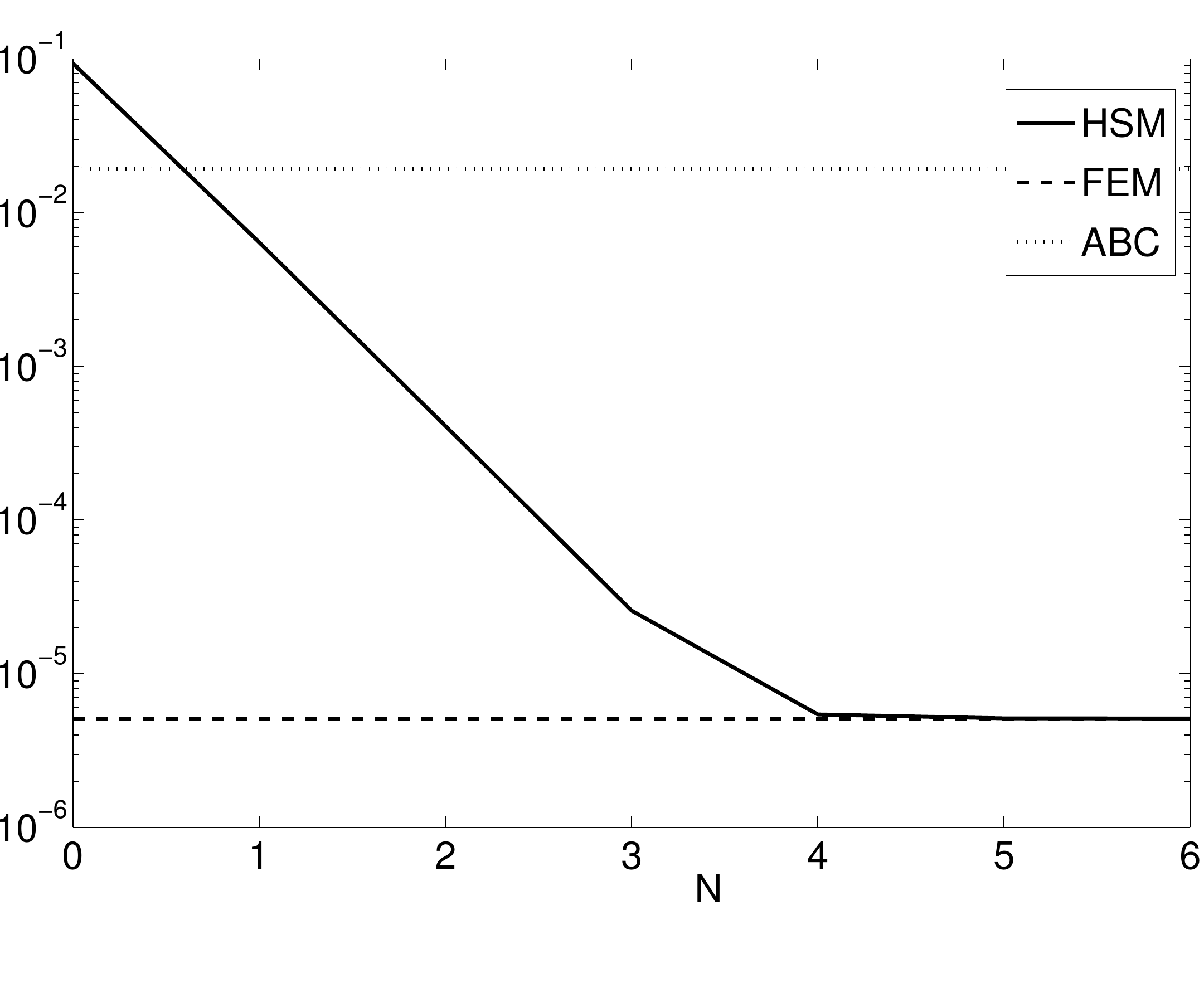}}
\subfigure{\includegraphics[width=0.3\textwidth]{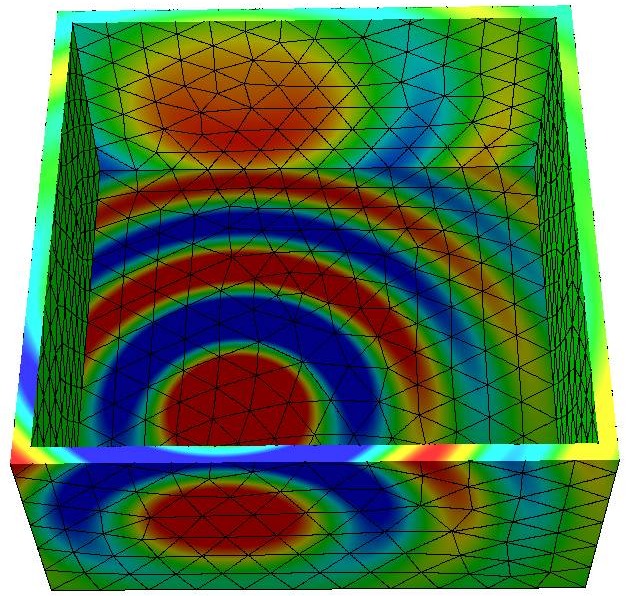}} \hspace{0.1\textwidth}
\subfigure{\includegraphics[width=0.45\textwidth]{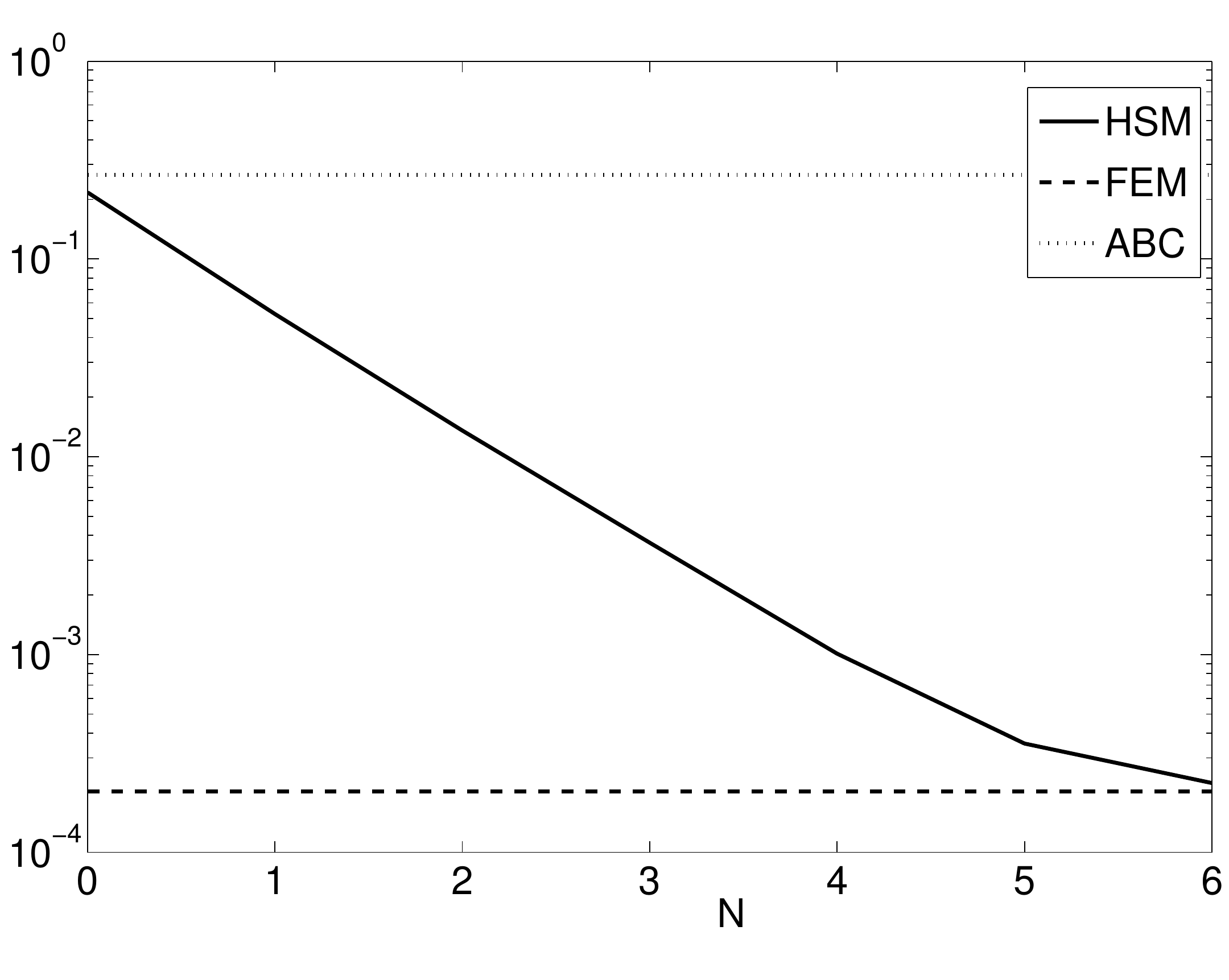}}
\caption{Right: $H(\curl, \Oi)$-error of the HSM w.r.t. the number $N$
of degrees of freedom in radial direction compared to the finite
element error and the error of a first order absorbing boundary
condition for the two different domains $\Oi$ on the left; Left: Cross-section of one Cartesian component of a
magnetic dipole}
\label{fig:error}
\end{figure}

In this example we resolve a magnetic dipole located at a point $y$
\begin{equation*} 
  E_y (x) = \nabla_x \times \left( \frac{e^{i \kappa
        |x-y|}}{|x-y|} \vector{1}{1}{1} \right)
\end{equation*} 
in two different geometries (see
Fig. \ref{fig:error}). $E_y$ is a radiating solution of
\eqref{eq:Maxwell} with given wavenumber $\kappa>0$ and an unbounded
domain $\Omega$, which does not contain the center $y$. \par

In the upper part of Fig. \ref{fig:error} the interior domain is the
intersection of two balls with radius $5$ and $6$ ( $\Oi=B_6 \setminus
B_5$), $\kappa=1$ and the center of the dipole is the origin
($y=0$). For the interior boundary $\partial B_5$ we use a Dirichlet
boundary condition given by the tangential part of $E_y$. For the
exterior boundary $\partial B_6$ we have three different cases: First
we again use the exact Dirichlet boundary condition in order to
compute the error of the finite element discretization of $\Oi$ with
polynomial order $6$. 
Second we use the first order absorbing boundary condition
$ \nabla \times E \times \nu - i \kappa E = 0 $
coming from the Silver-M\"uller radiation condition, which can be
applied without spending additional degrees of freedom. 
Last we use the Hardy space method with $\kappa_0=10$, reference point $V_0=(0,0,0)$ 
and $N=0,\dots,8$ , which leads to $102279$ ($N=0$) up to $229975$ ($N=8$)
degrees of freedom.

In the lower part of Fig. \ref{fig:error} the interior domain is the
intersection of two cubes $\Oi=[-3.2,3.2]^3 \setminus [-3,3]^3$, $\kappa=5$ and the
dipole is located at $y=(1,1.5,-1)^T$. The finite element method needs
$555480$ degrees of freedom for polynomial order $5$; together with
the Hardy space method with $\kappa_0=12.5$ and $V_0=(0,0,0)^T$ we
get $584642$ up to $1080374$ degrees of freedom. 

Both cases show a fast convergence of the Hardy space method, such that 
setting $N=4$ ($N=6$) suffices to reach the finite element error.

\subsection{Cavity resonances}

\begin{figure}
 \centering
\subfigure{\includegraphics[width=0.4\textwidth]{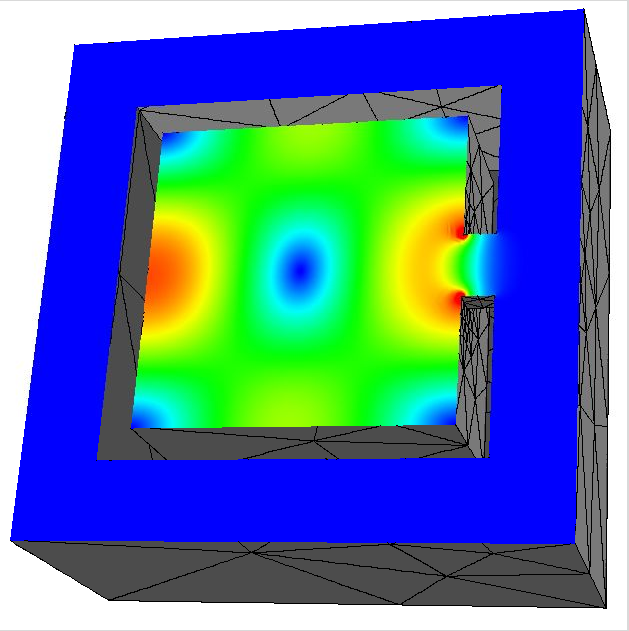}}\hfill
\subfigure{\includegraphics[width=0.4\textwidth]{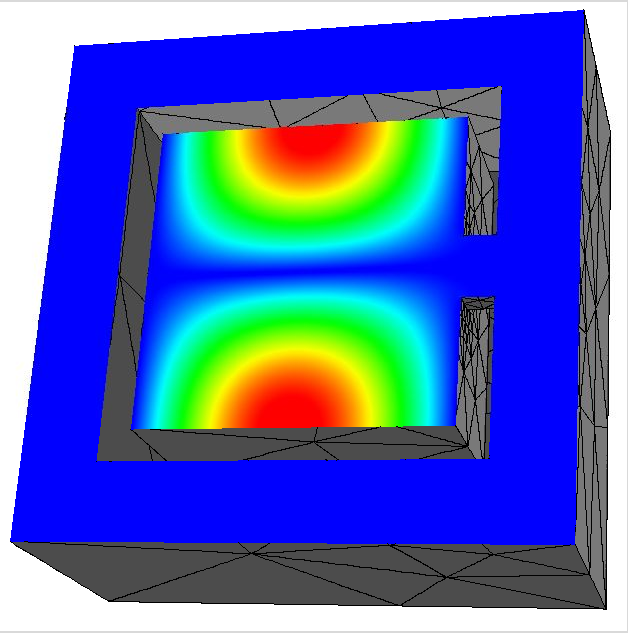}}
\caption{Cross-section of the absolute value of the two resonance
functions of the resonances close to $\kappa =\sqrt{3} \frac{\pi}{2} \approx 2.72$}
\label{fig:MaxResFunc}
\end{figure}

Here we search for resonances $\kappa\neq 0$ and
radiating electric fields $E \neq 0$ solving \eqref{eq:Maxwell} for
$\Omega = \setR^3 \setminus K$ and
$K=[-1.2,1.2]^3 \setminus \left([-1,1]^3 \cup [1,1.2]\times [-0.2,0.2]^2 \right)$.
Additionally, $E$ has to satisfy at $\partial K$ the perfectly
conducting boundary conditions
$
E \times \nu = 0
$
with $\nu$ being the outward normal vector. This problem is an
extension of the two dimensional acoustic open cavity problem
in~\cite{HohageNannen:09}. A similar acoustic problem is treated in~\cite{Marburg:06} 
with boundary element methods. 

In Fig.~\ref{fig:MaxResFunc} the absolute value of two resonance
functions on a cross-section of the interior domain
$\Oi= [-1.7,1.7]^3 \cap \Omega$ is shown. 
For a closed cavity ($\Omega=[-1,1]^3$), the
resonances are positive and analytically given by (see
\cite{AdamArbenzGeus:97})
$ \kappa=\sqrt{l +m + n } \frac{\pi}{2}$ for 
$l,m,n\in \setN_0 $ such that $ lm+ln+mn >0$
The resonance functions in Fig.~\ref{fig:MaxResFunc} belong to
resonances close to the second cavity resonance $\kappa =\sqrt{3}\frac{\pi}{2}$ 
with multiplicity $2$.

\begin{figure}
 \centering
\includegraphics[width=0.9\textwidth]{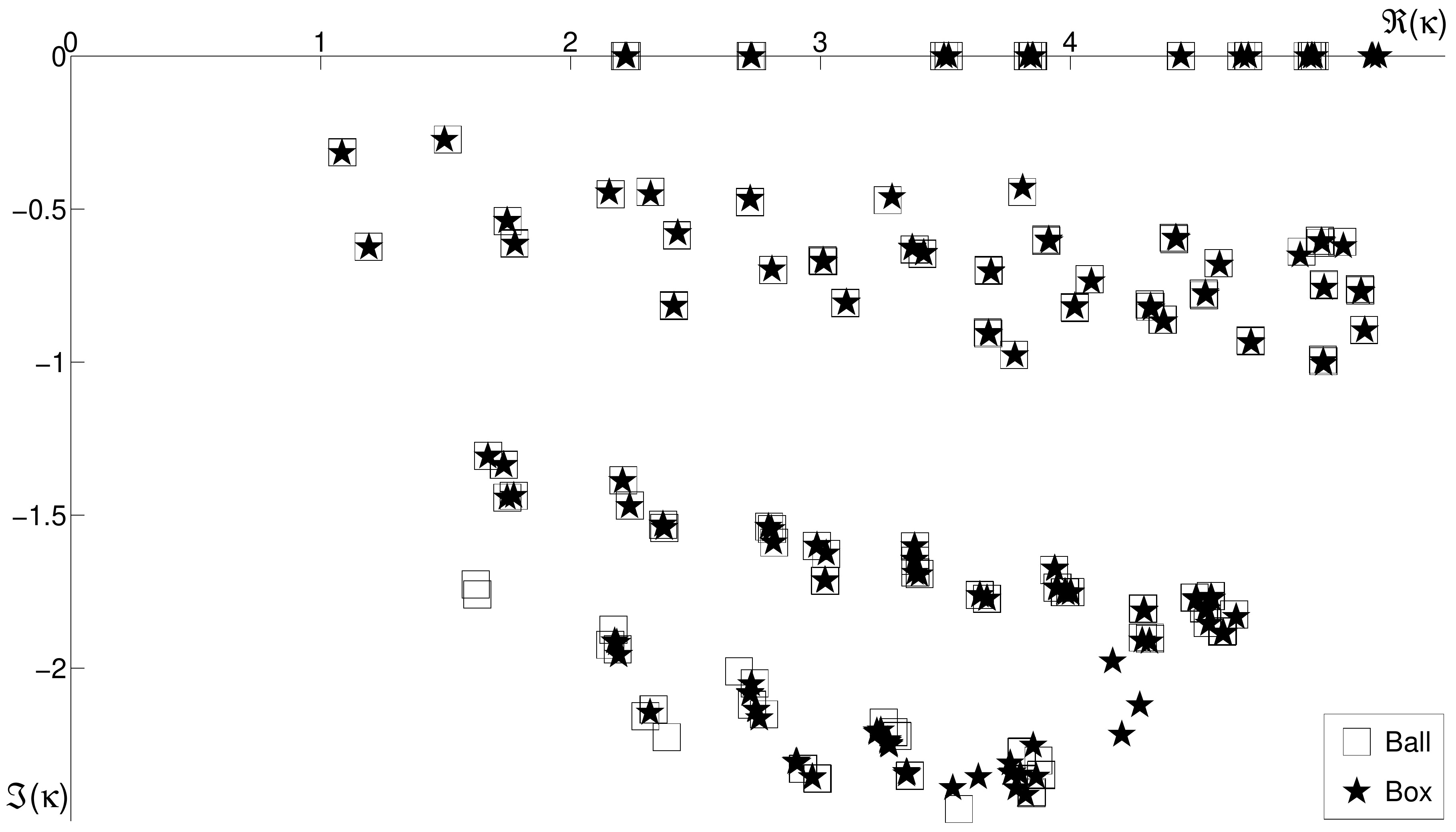}
\caption{Resonances of an open cavity for two different discretizations: 
  First $\Oi=B_{2.5} \cap \Omega$, $\kappa_0=5$, $N=10$ and finite element order $5$. 
  Second $\Oi= [-1.7,1.7]^3 \cap \Omega$, $\kappa_0=4-i$, $N=6$ and finite element order $6$}
\label{fig:MaxRes}
\end{figure}

Fig.~\ref{fig:MaxRes} shows the real and imaginary part of the
computed resonances for two different discretizations. For the first
we use the domain $\Oi=B_{2.5} \cap \Omega$ with
$\kappa_0=5$, $N=10$ and finite element order $5$ and in total
$358924$ degrees of freedom. For the second dicretization $\Oi=
[-1.7,1.7]^3 \cap \Omega$, $\kappa_0=4-i$, $N=6$ and finite element
order $6$ lead to $390548$ degrees of freedom. Both discretizations
give similar results for the cavity resonances near the real axis. The
multiplicity of the resonances is not visible in
Fig.~\ref{fig:MaxRes}. It is the same as expected from the resonances of
the closed cavity given in \cite{AdamArbenzGeus:97}. The exterior
resonances with in absolute values larger imaginary parts are mostly
identical for the two discretizations, but for the resonances at the
bottom of Fig.~\ref{fig:MaxRes} the discretizations are too coarse.

\subsection{Resonances of GaAs pyramidal micro-cavities}
\label{sec:ResPyramid}
A second example of cavity resonances is taken from~\cite{Karletal:09}.
The cavity is a pyramid with height $0.14142$
and a quadratic base of length $0.28284$ which is turned up-side down
and mounted on top of an infinite pyramid. Choosing the apex of the pyramids as 
reference point $V_0 = (0,0,0)$, the infinite pyramid is bounded by 
the infinite rays in direction $(1,1,-1)^T$, $(-1,1,-1)^T$, $(-1,1,-1)^T$ and
$(-1,-1,-1)^T$.  The computational domain $\Oi$ is a cuboid given by the 
vertices $(-0.2, -0.2, -0.1)$ and $(0.2,0.2,0.2)$. 
The pyramids are made of GaAs ($\ep=12.25$). In contrast to the first
example the exterior domain outside the plotted interior domain
consist of two different materials, namely air ($\ep=1$) and the
infinite GaAs pyramid ($\ep=12.25$). 

\begin{figure}
\centering
\subfigure{\includegraphics[width=0.45\textwidth]{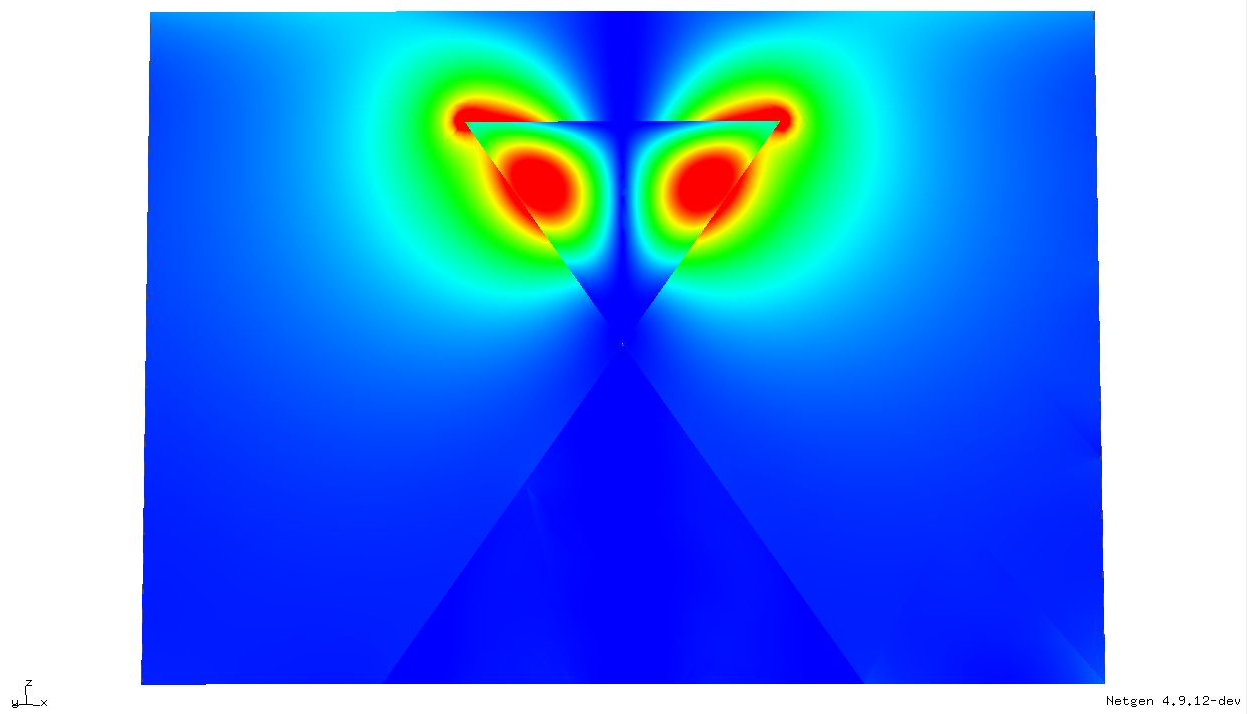}} 
\hfill
\subfigure{\includegraphics[width=0.54\textwidth]{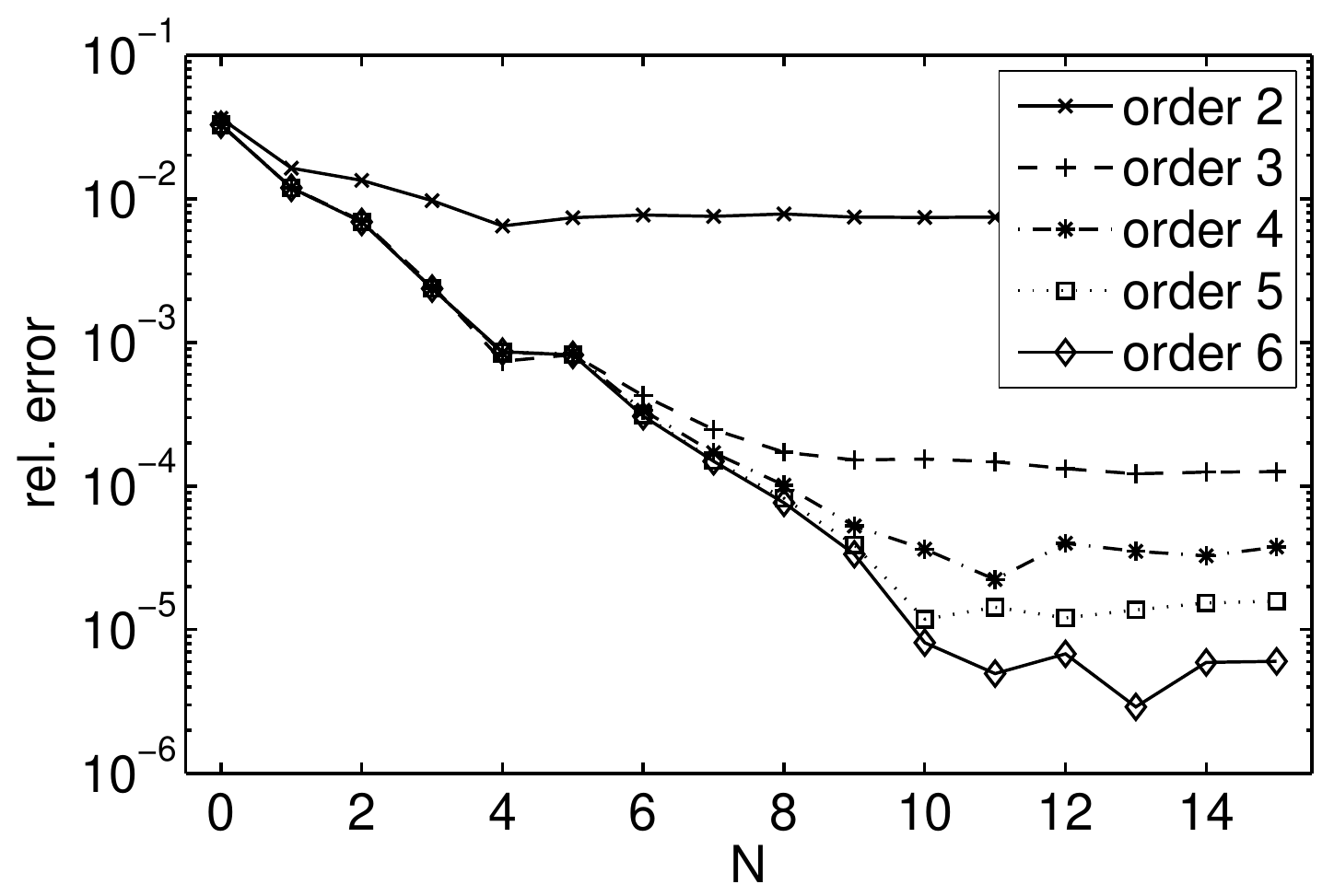}}
\caption{Left: Cross-section of the field intensity of the resonance
function to the resonance $\kappa \approx  12.3034  - 0.8816 i$. Right:
Relative error of the computed resonance w.r.t. the number of degrees
of freedom in radial direction for different finite element orders.}
\label{fig:ResInvPyr}
\end{figure}

In Fig.~\ref{fig:ResInvPyr} a cross-section of the field intensity of
the resonance field for the resonance $\kappa \approx 12.3034 - 0.8816
i$ is given.

To compute a reference solution $\Oi$ is discretized by 2477
tetrahedrons and finite elements of order $7$ leading to approximately
1.5 million unknowns for the interior domain. For the exterior domain
$18$ degrees of freedom in the radial direction cause in total 2
million unknowns. The error in Fig.~\ref{fig:ResInvPyr} is the
relative error of the computed resonance with respect to the reference
resonance $\kappa \approx 12.3034 - 0.8816 i $ for different finite
element orders and different numbers of degrees of freedom in radial
direction. The results indicate that depending on the finite element
order only 8 to 10 degrees of freedom in radial direction are needed.

\subsection{$(H(\divo),L^2)$ resonance test}
\label{sec:hdiv}

\begin{figure}
 \centering
\includegraphics[width=0.9\textwidth]{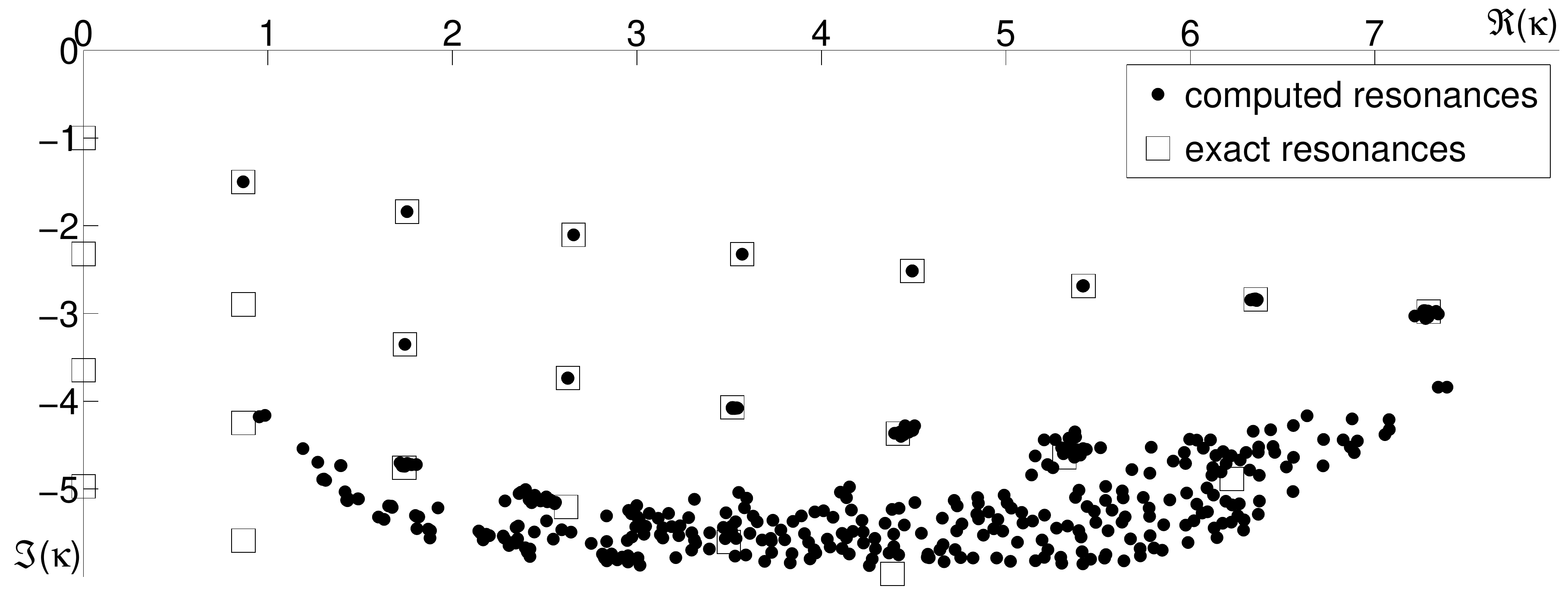}
\caption{computed and exact acoustic resonances with $\Omega=\setR^3 \setminus B_1$,  
$\Oi=[-1.2,1.2]^3 \setminus B_1$, $\kappa_0=5-i$, reference point $V_0=0$ and $N=15$}
\label{fig:MixedBallRes}
\end{figure}

In the last numerical test we compute acoustic resonances $\kappa$
outside a sphere with radius $1$. In this case the exact resonances
are given by the roots of the spherical Hankel functions of the first
kind, see~\cite[Example 3.24]{Nannen:08}. The multiplicity of a
resonance is $2n+1$ if the resonance is the root of the $n$th Hankel
function. Instead of solving the Helmholtz
problem~\eqref{eq:Helmholtz}, we solve the mixed formulation
\eqref{eq:Helmholtzdiv} and use the $H(\divo)$ and $L^2$ elements of
Sections~\ref{sec:Hdivelem} and~\ref{sec:Ltelem}.

\begin{figure}
 \centering
\subfigure{\includegraphics[width=.35\textwidth]{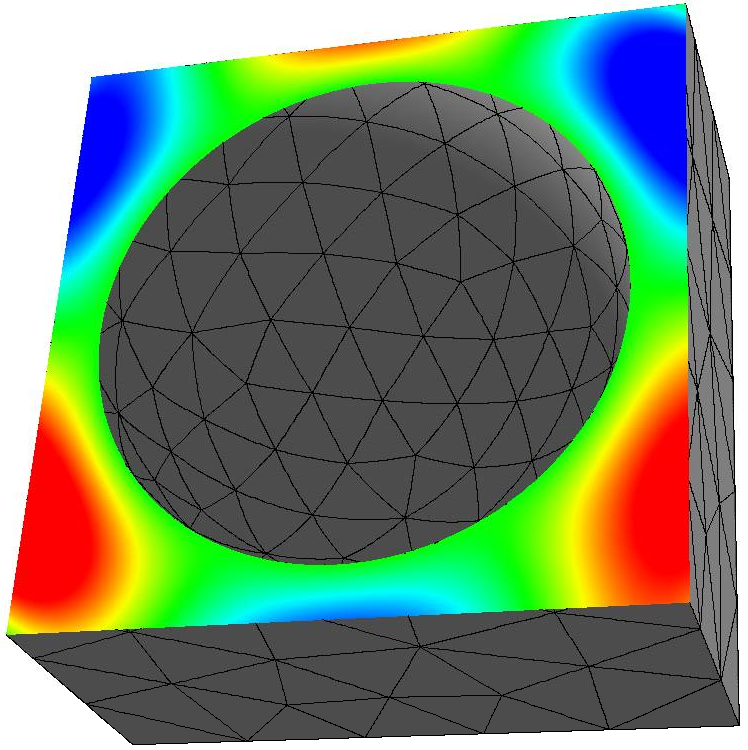}}\hfill
\subfigure{\includegraphics[width=.6\textwidth]{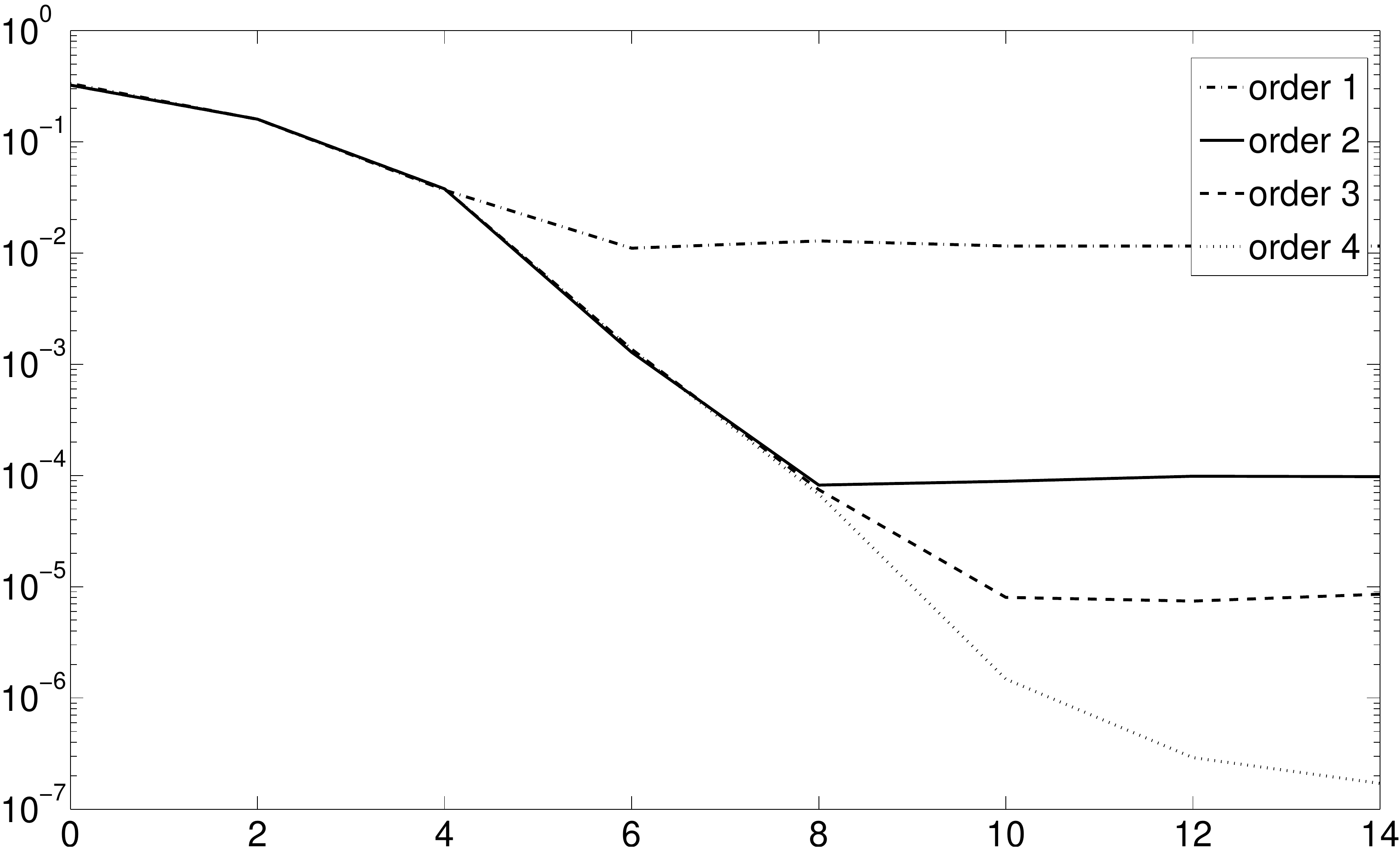}}
\caption{Left: Cross-section of one of the resonance functions to the resonance 
  $\kappa \approx 1.754 -1.839i$; right: relative error for different finite element 
  orders and different $N$ of the computed resonance against the analytical one with 
  $\Oi=[-1.2,1.2]^3 \setminus B_1$, $\kappa_0=7-0.3 i$ and reference point $V_0=0$}
\label{fig:MixedBallConv}
\end{figure}

In Fig. \ref{fig:MixedBallRes} the computed resonances $\kappa$ for
$\Oi=[-1.2,1.2]^3 \setminus B_1$, finite element order $3$,
$\kappa_0=5-i$, reference point $V_0=0$ and $N=15$ are compared to the
analytical ones. Again, for the resonances with in absolute values
larger imaginary part the discretization with in total $397316$
degrees of freedom is too coarse.

Fig. \ref{fig:MixedBallConv} shows the relative error of one computed
resonance against one of the roots of the third spherical Hankel
function of the first kind. In comparison to the results of
Sec.~\ref{sec:MagnDipole} and Sec.~\ref{sec:ResPyramid} more degrees
of freedom in radial direction are needed, since the Hardy space
method has to resolve the Hankel function. In total $27256$ (order
$1$, $N=0$) up to $549920$ (order $4$, $N=14$) degrees of freedom are
used.

\section*{Acknowledgement}
We gratefully acknowledge financial support from Deutsche
Forschungsgemeinschaft (DFG) in programs HO 2551/5-1 
and Matheon D9. Some parts of this work were developed at the Center for Computational
Engineering Science (CCES) of RWTH Aachen University in summer
2009.

\bibliographystyle{abbrv} \bibliography{bibliography}
\end{document}